\documentclass[reqno]{amsart}     
\usepackage{dsfont}
\usepackage{amsmath}
\usepackage{amsthm}
\usepackage{mathrsfs}
 \usepackage{bbm}
\usepackage{a4wide} 
\usepackage[T1]{fontenc}
\usepackage[latin1]{inputenc}
\usepackage{marginnote}
\usepackage{tikz}
\usepackage{pgfplots}
\usepackage{pgf}
\usetikzlibrary{arrows,automata}
\usetikzlibrary{positioning}
\usetikzlibrary{shapes.geometric}
\usetikzlibrary{shapes.arrows}
\usetikzlibrary{matrix}

\usepackage{xcolor}
\usepackage{caption}
\usepackage{hyperref}
\usepackage{array}
\usepackage{graphicx}

\usepackage[color,matrix,arrow]{xy}

\renewcommand{\(}{\left(}
\renewcommand{\)}{\right)}
\newtheorem{thm}{Theorem}
\newtheorem{prop}{Proposition}
\newtheorem{lemma}{Lemma}
\newtheorem{cor}{Corollary\!\!}

\newtheorem{ncor}{Corollary}

\theoremstyle{definition}

\newtheorem{df}{Definition}
\newtheorem{ex}{Example}

\theoremstyle{remark}

\newtheorem{rem}{Remark\!\!}
 
\newtheorem{nrem}{Remark}

\renewcommand{\Pr}{\mathbb{P}}

\newcommand{\bth}{\begin{thm}} 
\newcommand{\bl}{\begin{lemma}} 
\newcommand{\el}{\end{lemma}} 
\newcommand{\bp}{\begin{prop}} 
\newcommand{\ep}{\end{prop}} 
\newcommand{\bdf}{\begin{df}} 
\newcommand{\edf}{\end{df}} 
\newcommand{\brem}{\begin{rem}} 
\newcommand{\erem}{\end{rem}} 
\newcommand{\bnrem}{\begin{nrem}} 
\newcommand{\enrem}{\end{nrem}} 
\newcommand{\bex}{\begin{ex}} 
\newcommand{\eex}{\end{ex}} 
\newcommand{\bcor}{\begin{cor}} 
\newcommand{\ecor}{\end{cor}} 
\newcommand{\bncor}{\begin{ncor}} 
\newcommand{\encor}{\end{ncor}} 
\newcommand{\bpf}{\begin{proof}} 
\newcommand{\epf}{\end{proof}}

\newcommand{\D}{\displaystyle}

\begin{document} 

\title{Counting embeddings of rooted trees into families of rooted trees} 
\author[B. Gittenberger, Z. Go{\l}\k{e}biewski, I. Larcher, M. Sulkowska]{Bernhard Gittenberger,
Zbigniew Go{\l}\k{e}biewski, \\ Isabella Larcher, \and Ma{\l}gorzata Sulkowska} 
\thanks{This research has been supported by the \"OAD, grant PL04-2018, and Wroc{\l}aw University of Science and Technology grant 0401/0052/18.
} 
\address{Department of Discrete Mathematics and Geometry, Technische  
Universit\"at Wien, Wiedner Hauptstra\ss e 8-10/104, A-1040 Wien, Austria.}
\email{gittenberger@dmg.tuwien.ac.at} 
\address{
Department of Fundamentals of Computer Science, Wroc{\l}aw University of Science and Technology, ul. Wybrze{\.z}e
Wyspia{\'n}skiego 27, 50-370, Wroc\l aw, Poland.}
\email{zbigniew.golebiewski@pwr.edu.pl}
\address{Department of Discrete Mathematics and Geometry, Technische
Universit\"at Wien, Wiedner Hauptstra\ss e 8-10/104, A-1040 Wien, Austria.}
\email{isabella.larcher@tuwien.ac.at}
\address{
Department of Fundamentals of Computer Science, Wroc{\l}aw University of Science and Technology, ul. Wybrze{\.z}e
Wyspia{\'n}skiego 27, 50-370, Wroc\l aw, Poland.\newline
\hspace*{8pt} Universit{\'e} C{\^o}te d'Azur, CNRS, Inria, I3S, France.}
\email{malgorzata.sulkowska@pwr.edu.pl}


\date{\today}

\begin{abstract} 
The number of embeddings of a partially ordered set $S$ in a partially ordered set $T$ is the
number of subposets of $T$ isomorphic to $S$. If both, $S$ and $T$, have only one unique maximal
element, we define good embeddings as those in which the maximal elements of $S$ and $T$ overlap.
We investigate the number of good and all embeddings of a rooted poset $S$ in the family of all
binary trees on $n$ elements considering two cases: plane (when the order of descendants matters)
and non-plane. Furthermore, we study the number of embeddings of a rooted poset $S$ in the family
of all planted plane trees of size $n$. We derive the asymptotic behaviour of good and all embeddings in all cases and we prove that the ratio of good embeddings to all is of the order $\Theta(1/\sqrt{n})$ in all cases, where we provide the exact constants. Furthermore, we show that this ratio is asymptotically non-decreasing with $S$. Finally, we comment on the case when $S$ is disconnected.
\end{abstract} 
 
\maketitle

\section{Introduction}

This paper studies the number of embeddings of a given rooted tree in the family of (plane and
non-plane) binary trees, as well as planted plane trees. Here, the notion of embedding is wider
than just a copy. We assume the investigated structures to be partially ordered sets (in short:
posets) and by saying that there exists an embedding of $S$ into $T$ we understand (in the non-plane
case) that a poset $S$ is a subposet of $T$. We distinguish between good embeddings in which the
roots of $S$ and $T$ overlap and bad embeddings in which they do not. The number of good and bad
embeddings of a rooted structure in a complete binary tree was first investigated by Morayne~\cite{Morayne_complete}. His research was motivated by optimal stopping problems. The ratio of the number of good embeddings to the number of all embeddings and its monotonicity properties were used in estimates of conditional probabilities needed to obtain an optimal policy for the best choice problem considered on a complete (balanced) binary tree. This and similar results first served just as tools but soon became interesting questions about the structural features of posets on their own and resulted in a series of self-standing papers \cite{KLM_ratio_incr,KLM_AtoB_is_2tol,Georgiou}. Counting chains and antichains in trees took a special place in this pool \cite{Kuchta_chain1,Kuchta_chain2,Kubicki_chains_antichains}.

In this paper we present a follow-up and generalization of the results obtained by Kubicki
\emph{et al.}~\cite{KLM_ratio_incr,KLM_AtoB_is_2tol} and Georgiou~\cite{Georgiou}. We give the
asymptotic behaviour of the number of good and all embeddings of a rooted tree $S$ in the family
of plane (when the order of descendants matters) and non-plane binary trees, as well as planted
plane trees, on $n$ vertices. We prove that the ratio of the number of good embeddings to the
number of all embeddings is of the order $\Theta(1/\sqrt{n})$ in all cases and provide the exact
constants. Furthermore, we show that this ratio is asymptotically non-decreasing in $S$. 
We comment also on the case where $S$ is disconnected, \textit{i.e.} a forest. In
order to obtain those results we use tools of analytic combinatorics that have not been used
before in the aforementioned papers.

The results of our paper may also be put into the framework of counting patterns in large
structures. This is a vast field where many different types of structures have been considered.
We only mention subgraph avoidance (and characterizing whole graph classes like series-parallel or
planar graphs in that way) or subgraph counts in random graphs (see \cite{JLR00, AS08}),
pattern avoidance in permutations (see \cite{B12}) or in trees (see \cite{D09}), or pattern
avoidance in lattice paths and words, where many particular patterns have been treated separately
(see \cite{D99} or the introduction of \cite{ABBG20} for a survey) and eventually put under a
unifying umbrella in \cite{ABBG20}.

The closest to the present work is pattern counting in trees.  One of the earliest investigations
of this kind was \cite{RS75}, where the enumeration of given stars as subgraphs in trees
(equivalently nodes of fixed degree) was treated. Later generalizations are found in \cite{DG99,
PS12} (multivariate setting), in \cite{li2016asymptotic} (distinct patterns) or \cite{G06} (large
patterns of that type). A method to deal with general contiguous patterns in trees by means of
generating functions was developed in \cite{CDKK08}, which was partially generalized to planar
maps recently \cite{DY18, CDK19, DS20}. Pattern avoidance in trees was the topic of \cite{R10},
where also the concept of Wilf equivalence was dealt with, which was adopted from pattern
avoidance in permutations. 

Except for permutations, where most of the patterns that have been studied so far are non-contiguous, the considered patterns in other domains are 
typically contiguous. To our knowledge, the first work considering non-contiguous patterns in
trees is \cite{DPTW12}. In the present paper, the tree which is embedded becomes in general a
collection of (partially) non-adjacent nodes in the tree where it is embedded. It can therefore
be seen as a non-contiguous pattern occurring in that tree. Thus, our paper deals with certain
enumeration problems for non-contiguous patterns in trees. 

\medskip
The paper is organized as follows. Section \ref{sec_def_not} introduces basic definitions and notation. Section~\ref{sec_stop} provides possible applications of our
results in optimal stopping problems. In Section~\ref{sec_gen_functions} we obtain generating
functions for the number of good and the number of all embeddings of a rooted tree in the family
of all plane binary trees with a given number of vertices. Section~\ref{sec_asymptotics} is
devoted to the asymptotics of this number when the size of the underlying tree is tending to infinity.
Moreover, we investigate the asymptotics as well as the asymptotic monotonicity of their ratio.
In Section~\ref{sec_disconnected} we briefly discuss the case when the embedded structure is
disconnected, \textit{i.e.} it is a forest.  Section~\ref{sec_non_plane} deals with the non-plane
binary case and in Section~\ref{sec_planted_plane} the problem is extended to planted plane trees.
A discussion of the obtained results as well as an outlook into some related future problems is
given in Section~\ref{sec_discussion}.

\section{Definitions and notation} \label{sec_def_not}

By ${\mathcal{B}}_n$ we denote the family of (unlabelled) plane binary trees with $n$ nodes. A
binary tree is a tree in which each node has either $0$ or $2$ descendants and by plane we
understand that the order of subtrees of a given node matters, \textit{i.e.} we distinguish
between the different embeddings of a tree in the plane. It is commonly known that for odd $n$
the cardinality $|{\mathcal{B}}_n|$ of ${\mathcal{B}}_n$ satisfies $|{\mathcal{B}}_n|=
\boldsymbol{C}_{\frac{n-1}{2}}$, where $\boldsymbol{C}_k$ is the $k$-th Catalan number given by
$\boldsymbol{C}_k = \frac{1}{k+1}{2k \choose k}$.  Note that all binary trees have odd sizes and
thus, for even $n$ the cardinality $|\mathcal{B}_n|$ is zero.  All plane binary trees of size
$5$ are shown in Figure \ref{fig_bin_5}. We assume also that all edges are directed towards the
descendants. Therefore, the in-degree of the root, as well as the out-degree of each leaf, is
always $0$. A vertex is said to be $d$-ary if its out-degree equals $d$. Subsequently, the root
of a tree $T$ will be denoted by $\mathds{1}_T$.

By ${\mathcal{V}}_n$ we denote the family of non-plane binary trees with $n$ nodes. By non-plane
we understand that the subtrees of a given node are treated as a set of subtrees, \textit{i.e.}
there is no ordering. E.g., there is only one non-plane binary tree of size $5$, see Figure
\ref{fig_bin_5}. Again for even $n$ the cardinality $|\mathcal{V}_n|$ is zero. For odd $n$ the values $|\mathcal{V}_n|$ are known as Wedderburn-Etherington numbers and do not have a closed form ($|\mathcal{V}_1|=1, |\mathcal{V}_3|=1, |\mathcal{V}_5|=1, |\mathcal{V}_7|=2, |\mathcal{V}_9|=3, \ldots$).

Planted plane trees (also known as Catalan trees) are rooted plane trees where each internal node
can have arbitrarily many descendants. We denote the family of planted plane trees of
size $n$ by $\mathcal{T}_n$. For all $n$ the cardinality of $\mathcal{T}_n$ satisfies $|\mathcal{T}_n| = \boldsymbol{C}_{n-1}$.

\begin{figure}[!h]
\centering
\includegraphics[scale=0.85]{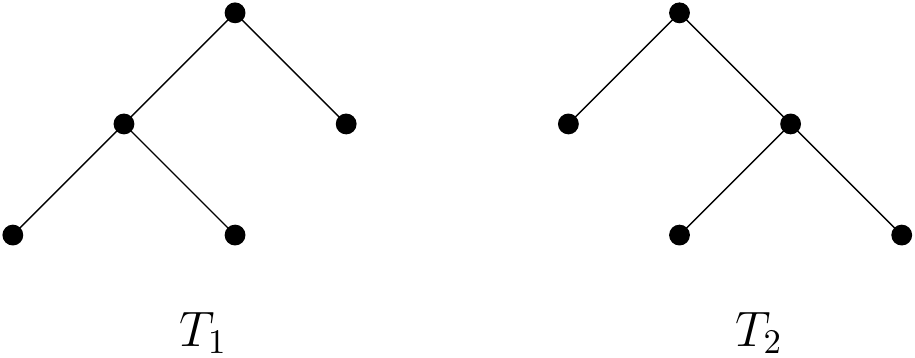}
\caption{\footnotesize The family $\mathcal{B}_5 = \{T_1, T_2\}$ of plane binary trees is of size $|\mathcal{B}_5| = \boldsymbol{C}_2 = 2$, while the family $\mathcal{V}_5 = \{T_1\}$ of non-plane binary trees has the size $|\mathcal{V}_5| = 1$.}
\label{fig_bin_5}
\end{figure}

This paper concentrates on investigating the number of embeddings of any rooted tree (or a forest of rooted trees - a disconnected graph whose components are rooted trees) in all trees from either family $\mathcal{B}_n$, $\mathcal{V}_n$ or $\mathcal{T}_n$. 
An embedding of a rooted tree $S$ into another rooted tree $T$ can be seen as a kind of
generalized pattern occurrence of $S$ in $T$, defined as follows, where we distinguish between the plane and the non-plane case.

\begin{df}[non-plane embedding]
Let $S$ and $T$ be two non-plane rooted trees.
When interpreting $T$ as the cover graph of a partially ordered set (poset), rooted at the root of
$T$, \textit{i.e.} at the single maximal element of the poset, then an embedding of $S$ into $T$ can be defined as any subposet of $T$ isomorphic to $S$. 
\end{df}

\begin{nrem}
Note that there exists a non-plane embedding of a binary tree $S$ into a binary tree $T$ if and only if $S$ is a minor of $T$.
\end{nrem}

\begin{nrem} 
Instead of starting from a tree as combinatorial structure and then interpreting it as a poset,
we may also start from posets and then define a \emph{tree poset} as a poset $P$ which has exactly
one maximal element and such that any Hasse diagram of $P$ looks like a (combinatorial) tree.
This is equivalent to the definition of a tree poset given in \cite{F20+}.\footnote{For the sake
of better distinction from a combinatorial tree, we use the term ``tree poset'' for what is simply
called ``tree'' in \cite{F20+}.} Likewise, an embedding of a tree poset $S$ into another tree
poset $T$, as defined in \cite{F20+}, matches exactly the definition of a non-plane embedding
given above. 
\end{nrem}

\begin{df}[plane embedding]
Let $S$ and $T$ be two plane rooted trees.
If we interpret $T$ to be a Hasse diagram of a poset, then an embedding of $S$ into $T$ can be
defined as any subposet of $T$ isomorphic to $S$ in which the left-to-right order of the children
of each node of $S$ is inherited from $T$ (thus, a plane version of a subposet). 
\end{df} 

\begin{nrem}
So, in the plane case $S$ and $T$ can be interpreted as Hasse diagrams of posets, and whenever $S$
can be embedded in $T$ it follows that $S$ is a subposet of~$T$. However, note that the respective
posets can possibly be represented as different Hasse diagrams in such a way that no embedding of 
the corresponding trees is possible. 
\end{nrem}

We say that an embedding of $S$ into $T$ is \textit{good} if it contains the root of $T$.
Otherwise we call it a \textit{bad embedding}. If there exists at least one embedding of $S$ into $T$, we write $S \subseteq T$. All embeddings of a cherry, (\textit{i.e.} a tree composed only of a root and its two children) in a given binary tree of size $5$ are given in Figure \ref{cherry_5}. Four of them are good and the last one is bad.

\begin{figure}[!ht]
\centering
\includegraphics[scale=0.7]{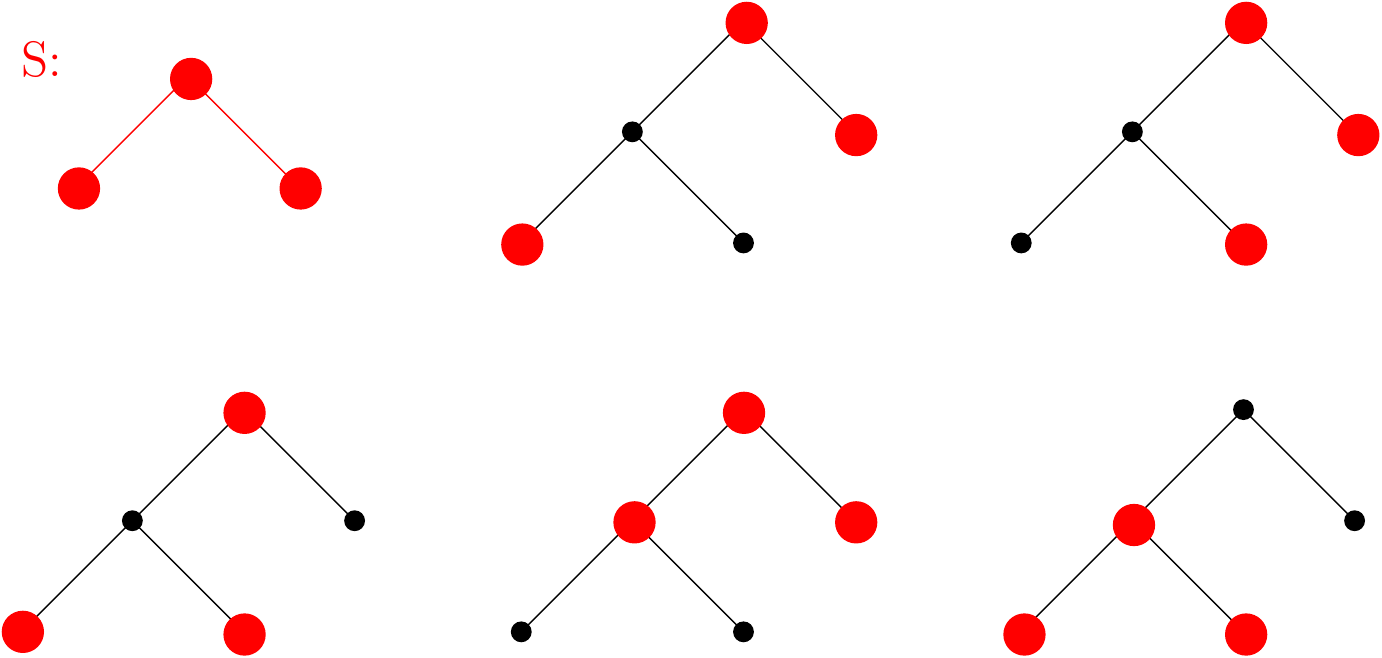}
\caption{\footnotesize All five embeddings of a cherry $S$ in a given plane binary tree of size~$5$. Or all four embeddings of a cherry $S$ in a given non-plane binary tree of size $5$, since in the non-plane case the two rightmost pictures in the upper row represent the same embedding (they can easily be mapped onto each other via a simple automorphism that changes the order of the two leftmost leaves).}
\label{cherry_5}
\end{figure}

Subsequently the size of the tree $S$ will always be denoted by $m$, while the size of~$T$ is consistently denoted by $n$. Thus, for the asymptotic analysis of the number of embeddings of a tree $S$ into a class of trees of size $n$, the quantity $m$ is considered to be a constant, while $n$ tends to infinity.

For $S$, the structure that we embed, we define its \textit{degree distribution sequence} as $d_S
= (d_0, d_1, \ldots, d_{m-1})$, where $d_i$ is the number of vertices in $S$ with out-degree equal
to $i$. Note that $d_0$ is simply the number of leaves, which will be, interchangeably, denoted by
$l$ (\textit{i.e.} $l=d_0$). Similarly, $d_1$ is the number of unary nodes, which will be,
interchangeably, denoted by $u$ (\textit{i.e.}~$u=d_1$). The number of all embeddings of a given
tree $S$ in $T$ will be denoted by $a_{T}(S)$ and the number of its good embeddings in $T$ by
$g_{T}(S)$. The number of all embeddings of $S$ in a family $\mathcal{F} = \{F_1, \ldots, F_N \}$
will be denoted by $a_{\mathcal{F}}(S)$ and understood as the cumulative number of embeddings of
$S$ into all elements of $\mathcal{F}$, \emph{i.e.} $a_{\mathcal{F}}(S) = \sum_{i=1}^{N} a_{F_i}(S)$. Analogously, we define the number of good embeddings of $S$ in $\mathcal{F}$: $g_{\mathcal{F}}(S) = \sum_{i=1}^{N} g_{F_i}(S)$. For $S$ being a cherry and $\mathcal{B}_5 = \{T_1, T_2\}$, we obtain $a_{T_1}(S) = a_{T_2}(S) = 5$, $g_{T_1}(S) = g_{T_2}(S) = 4$, thus $a_{\mathcal{B}_5}(S) = 10$ and $g_{\mathcal{B}_5}(S) = 8$ (compare Figure \ref{cherry_5}). Notations for ${\mathcal{V}_n}$ and $\mathcal{T}_n$ are analogous.

Throughout this paper we use the standard notation $f(n)\sim g(n)$ if $\lim_{n \rightarrow \infty}  \frac{f(n)}{g(n)} = 1$. Also $[z^n]f(z)$ denotes the coefficient of $z^n$ in the formal power series $f(z) = \sum_{n \geq 0} f_n z^n$, \textit{i.e.} $[z^n]\left(\sum_{n \geq 0} f_n z^n\right) = f_n$.

\section{Applications in optimal stopping problems} \label{sec_stop}
The most prominent problem in the area of optimal stopping is the so-called 
``secretary problem'' (consult~\cite{Lindley,Ferguson,freeman1983secretary,S91}), where one assumes a linear order on the applicants for a secretary position concerning their qualifications. 
The applicants are interviewed in a random order and the decision whether to hire an applicant has to be made immediately after the interview - a rejected applicant cannot be hired at a later point. Thus, if we interview all the candidates, we have to hire the last applicant.
The goal is to find the optimal stopping strategy to hire the best applicant. Thus, we want to stop at the time maximizing the probability that the present applicant is the best one overall, \textit{i.e.} the maximum element in the linear order. It has been proved (see for example \cite{Lindley, gilbert2006recognizing}) that for a large number of applicants it is optimal to wait until approximately $37 \%$ (more precisely $\frac{100}{e}  \%$) of the applicants have been interviewed and then to select the next relatively best one. This optimal algorithm returns the best applicant with asymptotic probability of $1/e$.
The secretary problem has been extended and generalized in many different directions. One of these is the
extension to partially ordered sets, possibly with more than one maximal element, see
\cite{Stadje,Gnedin}. Optimal strategies for particular posets were investigated among others in
\cite{Morayne_complete,K13}. Versions for unknown poset, when the selector knows in advance only
its cardinality, were presented in \cite{univ_poset,FW10,GM13}. Another interesting generalization
was to replace the underlying poset structure by a directed graph. This version was first
considered on directed paths by Kubicki and Morayne in \cite{dir_path} and later extended to other families of graphs and different versions of the game (consult \cite{S12,GMS15,kPaths}).

In the remainder of this section we give examples of stopping problems in which either the value $a_{\mathcal{V}_n}(S)$ or the ratio $g_{\mathcal{V}_n}(S)/a_{\mathcal{V}_n}(S)$ (both investigated in this paper) plays a crucial role in estimating the conditional probabilities needed to obtain the optimal policy. One can consider analogous examples for the families ${\mathcal{B}_n}$ or ${\mathcal{T}_n}$ as well.

Let us think about elements of $\mathcal{V}_n$ as of Hasse diagrams of posets. Consider the following process. Elements (\textit{i.e.} nodes) of some $T$ from $\mathcal{V}_n$ appear one by one in a random order (all permutations of elements of $T$ are equiprobable). At time $t$, \textit{i.e.} when $t$ elements have already appeared, the selector can see a poset induced on those elements. He knows that the underlying structure is drawn uniformly at random from $\mathcal{V}_n$.

\begin{ex}[Best choice problem for the family of binary trees]
The selector's task is to stop the process maximizing the probability that the element that has just appeared is the root of the underlying structure. He wins only if the chosen element is indeed $\mathds{1}_T$. Note that it neither pays off to stop the process when the induced structure is disconnected nor when the currently observed element is not the maximal one in the induced poset. The selector wonders whether to stop only if the emerged element at time $t$ is the unique maximal element in the induced structure. In order to take a decision whether to stop at time $t$, he needs to know the probability of winning if he stops now. Let $W_t$ denote the event of winning when stopping at time $t$, $S_t$ the event that at time $t$ he observes a certain structure $S$ with degree distribution sequence $d_S$ and $R_i$ denote the event that $T_i$ has been drawn as the underlying structure, where we use the notation $\mathcal V_n=\{T_1,\dots,T_N\}$ with $N=|\mathcal V_n|$. Then the probability of winning if he stops at time $t$ is given by
\[
\begin{split}
\Pr[W_t|S_t] & = \sum_{i=1}^{N} \Pr[W_t|S_t \cap R_i] \Pr[R_i|S_t] = \sum_{i=1}^{N} \frac{g_{T_i}(S)}{a_{T_i}(S)} \frac{\Pr[S_t|R_i]\Pr[R_i]}{\Pr[S_t]}.
\end{split} 
\]
Since $\Pr[R_i] = 1/N$, $\Pr[S_t|R_i] = a_{T_i}(S)/{n \choose t}$ and
$$
\Pr[S_t] = \sum_{i=1}^{N} \Pr[S_t|R_i] \Pr[R_i] = \sum_{i=1}^{N} \frac{a_{T_i}(S)}{{n \choose t}}\frac{1}{N} = \frac{a_{\mathcal{V}_n}(S)}{N {n \choose t}}
$$
we get
$$
\Pr[W_t|S_t] = \sum_{i=1}^{N} \frac{g_{T_i}(S)}{a_{T_i}(S)} \frac{a_{T_i}(S)}{{n \choose t}} \frac{1}{N} \frac{N {n \choose t}}{a_{\mathcal{V}_n}(S)} = \frac{g_{\mathcal{V}_n}(S)}{a_{\mathcal{V}_n}(S)}.
$$
\end{ex}

\begin{ex}[Identifying complete balanced binary trees]
The selector has to identify whether the underlying structure is a complete balanced binary tree or not. The payoff of the game, if he stops the process at time $t$, is $n-t$ if he guesses correctly and $0$ otherwise. He has to maximize the expected payoff. At moment $t$ he observes a structure $S$, which is not necessarily connected. Again, in order to make a decision whether to stop, he needs to know what is the probability that the currently observed structure is a subposet of a complete balanced binary tree. For a rooted tree $S$ this probability is given by
$$
\frac{a_{T_b}(S)}{a_{\mathcal{V}_n}(S)},
$$
where $T_b \in \mathcal{V}_n$ denotes the complete balanced binary tree of size $n$.
\end{ex}

\section{Generating functions for the number of embeddings in $\mathcal{B}_n$} \label{sec_gen_functions}
In this section we derive generating functions for the sequences $a_{\mathcal{B}_n}(S)$ and $g_{\mathcal{B}_n}(S)$, where $S$ is a given rooted plane tree of size $m$. In order to do so, we use the symbolic method (consult \cite{Flajolet_book}).

\begin{thm}
Consider a rooted tree $S$ with degree distribution sequence $d_S=(l,u,d_2, \dots, d_{m-1})$. The
generating function $A_S(z)$ of the sequence $a_{\mathcal{B}_n}(S)$, which counts the number of all embeddings of $S$ into all trees of the family $\mathcal{B}_n$ is given by
$$
{A_S}(z) = \left( \frac{1}{1-2zB(z)}\right)^{m+l-1} z^{l+u-1} ~B(z)^{l+u}~ 2^u \prod_{i = 3}^{m-1} (\boldsymbol{C}_{i-1})^{d_i},
$$
where $B(z)$ is the generating function of the family of plane binary trees, \textit{i.e.}
$$
B(z) = \frac{1-\sqrt{1-4z^2}}{2z} = \boldsymbol{C}_0 z + \boldsymbol{C}_1 z^3 + \boldsymbol{C}_2 z^5 + \boldsymbol{C}_3 z^7 + \ldots ~.
$$
\end{thm}

\begin{nrem}
Note that ${A_S}(z)$ depends only on the degree distribution sequence $d_S$, not the particular shape of $S$. Thus, as long as $d_{S_1}$ and $d_{S_2}$ are the same, $A_{S_1}(z)$ and ${A_{S_2}}(z)$ coincide even if $S_1$ and $S_2$ are not isomorphic. However, we use the subscript $S$ to provide a transparent notation.
Moreover, note that $A_S(z)$ does also depend on the tree class $\mathcal{B}_n$ in which we embed the tree~$S$. In order to avoid a large number of indices we will omit to indicate this dependence and just emphasize at this point that the generating functions $A_S(z)$ may differ according to the underlying tree classes. 
\end{nrem}

\begin{proof}
First let us recall that the class $\mathcal{B}$ of plane binary trees can be specified by
\[ \mathcal{B} = \{ \bullet \} + \{ \bullet \} \times \mathcal{B} \times \mathcal{B}, \]
since every node is either a leaf or a binary node with two attached binary trees. By means of the symbolic method (see \cite{Flajolet_book}) we can directly translate this specification into a functional equation that defines the generating function $B(z)$ of binary trees where $z$ marks the number of nodes, which gives
\[ B(z) = z + zB(z)^2.\]
Solving this equation for $B(z)$ yields the explicit formula that is given in the theorem.
Now, we start the proof of the expression of $A_S(z)$ with the case where $S$ is a Motzkin tree, \textit{i.e.} a tree where each internal node has either one or two children. The specification of the class $\mathcal M$ of Motzkin trees is 
\begin{equation}\label{motzkin}
\mathcal M=\{\bullet\}+\{\bullet\}\times \mathcal M+\{\bullet\}\times \mathcal M \times \mathcal M
\end{equation} 
and thereby we must distinguish between the three cases whether $S$ is a single node, or the root
of $S$ is a unary node, or a binary node, and hence falling into the respective
subclass of $\mathcal M$ among the subclasses that we find as summands on the right-hand side of
\eqref{motzkin}.
The generating function ${A_S}(z)$ for the number of embeddings of $S$ into the family 
$\mathcal{B}_n$ can then be recursively defined by
\begin{align}
\label{equ:GF_recusrive_A_S_planebinary}
A_{S}(z) = \begin{cases} zB'(z) & \quad \text{if} \ S = \bullet, \\ 
\frac{2 zB(z)}{1-2zB(z)} A_{\tilde{S}}(z) & \quad \text{if} \ S = (\bullet,\tilde{S}), \\
\frac{z}{(1-2zB(z))^2} {A_{{S_L}}}(z) {A_{{S_R}}}(z) & \quad \text{if} \ S = (\bullet,S_L,S_R), 
\end{cases} 
\end{align}
where the three cases correspond to the cases described above. 
The first case, which yields a factor $zB'(z)$, corresponds to marking a node in the underlying
tree $T$ (\textit{i.e.} pointing at a node), because obviously a single vertex can be embedded in
every node. 
We can also interpret it as counting the number of pairs $(T,E)$ where $E$ is an embedding of $S$
into $T$. 

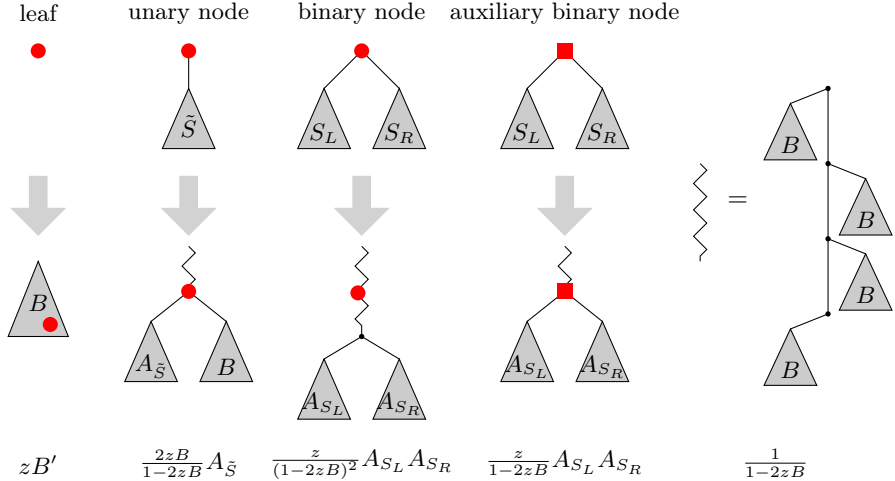
\begin{figure}[!h]
\begin{center}
\begin{tikzpicture}
\begin{scope}[xshift=-4cm]
	\draw (0,0.5) node {{\small leaf}};

	\fill[red] (0,0) circle (0.1);
	
	\node[single arrow, rotate=270, fill=gray!35, minimum height=0.8cm] at (0,-2) {};
	
	\filldraw[fill=gray!40] (0,-2.8) -- (-0.4,-3.8) -- (0.4,-3.8) -- cycle;
	\fill[red] (0.16,-3.64) circle (.1);
	\draw (0,-3.35) node {{\small $B$}};
	
	\draw (0,-5.5) node {{\small $z B'$}};
\end{scope}

\begin{scope}[xshift=-2cm]
	\draw (0,0.5) node {{\small unary node}};

	\draw (0,0) -- (0,-0.5);
	\fill[red] (0,0) circle (0.1);
	\filldraw[fill=gray!40] (0,-0.5) -- (-0.35,-1.3) -- (0.35,-1.3) -- cycle;
	\draw (0,-1) node {{\small $\tilde{S}$}};
	
	\node[single arrow, rotate=270, fill=gray!35, minimum height=0.8cm] at (0,-2) {};
	
	\draw[decorate, decoration=zigzag] (0,-2.6) -- (0,-3.2);
	\draw (0,-3.2) -- (0-0.5,-0.5-3.1);
	\draw (0,-3.2) -- (0+0.5,-0.5-3.1);
	\fill[red] (0,-3.2) circle (0.1);
	\filldraw[fill=gray!40] (0-0.5,-0.5-3.1) -- (-0.35-0.5,-1.3-3.1) -- (0.35-0.5,-1.3-3.1) -- cycle;
	\filldraw[fill=gray!40] (0+0.5,-0.5-3.1) -- (-0.35+0.5,-1.3-3.1) -- (0.35+0.5,-1.3-3.1) -- cycle;
	\draw (0-0.5,-1-3.2) node {{\small $A_{\tilde{S}}$}};
	\draw (0+0.5,-1-3.2) node {{\small $B$}};
	
	\draw (0,-5.5) node {{\small $\frac{2zB}{1-2zB} A_{\tilde{S}}$}};
\end{scope}

\begin{scope}[xshift=0.3cm]
	\draw (0,0.5) node {{\small binary node}};

	\draw (0,0) -- (0-0.5,-0.5);
	\draw (0,0) -- (0+0.5,-0.5);
	\fill[red] (0,0) circle (0.1);
	\filldraw[fill=gray!40] (0-0.5,-0.5) -- (-0.35-0.5,-1.3) -- (0.35-0.5,-1.3) -- cycle;
	\draw (0-0.5,-1.1) node {{\small $S_L$}};
	\filldraw[fill=gray!40] (0+0.5,-0.5) -- (-0.35+0.5,-1.3) -- (0.35+0.5,-1.3) -- cycle;
	\draw (0+0.5,-1.1) node {{\small $S_R$}};	
	
	\node[single arrow, rotate=270, fill=gray!35, minimum height=0.8cm] at (0,-2) {};
	
	\draw[decorate, decoration=zigzag] (0,-2.6) -- (0,-3.8);
	\fill[red] (-0.05,-3.22) circle (0.1);
	\fill[black] (0,-3.8) circle (1pt);
	\draw (0,-3.8) -- (0-0.5,-0.5-3.6);
	\draw (0,-3.8) -- (0+0.5,-0.5-3.6);
	\filldraw[fill=gray!40] (0-0.5,-0.5-3.6) -- (-0.35-0.5,-1.3-3.6) -- (0.35-0.5,-1.3-3.6) -- cycle;
	\filldraw[fill=gray!40] (0+0.5,-0.5-3.6) -- (-0.35+0.5,-1.3-3.6) -- (0.35+0.5,-1.3-3.6) -- cycle;
	\draw (0-0.5,-1-3.7) node {{\small $A_{S_L}$}};
	\draw (0+0.5,-1-3.7) node {{\small $A_{S_R}$}};
	
	\draw (0,-5.5) node {{\small $\frac{z}{(1-2zB)^2} A_{S_L} A_{S_R}$}};
\end{scope}

\begin{scope}[xshift=3cm]
	\draw (0,0.5) node {{\small auxiliary binary node}};

	\draw (0,0) -- (0-0.5,-0.5);
	\draw (0,0) -- (0+0.5,-0.5);
	\fill[red] (-0.1, -0.1) rectangle (0.1, 0.1);
	\filldraw[fill=gray!40] (0-0.5,-0.5) -- (-0.35-0.5,-1.3) -- (0.35-0.5,-1.3) -- cycle;
	\draw (0-0.5,-1.1) node {{\small $S_L$}};
	\filldraw[fill=gray!40] (0+0.5,-0.5) -- (-0.35+0.5,-1.3) -- (0.35+0.5,-1.3) -- cycle;
	\draw (0+0.5,-1.1) node {{\small $S_R$}};	
	
	\node[single arrow, rotate=270, fill=gray!35, minimum height=0.8cm] at (0,-2) {};
	
	\draw[decorate, decoration=zigzag] (0,-2.6) -- (0,-3.2);
	\draw (0,-3.2) -- (0-0.5,-0.5-3.1);
	\draw (0,-3.2) -- (0+0.5,-0.5-3.1);
	\fill[red] (-0.1,-3.3) rectangle (0.1, -3.1);
	\filldraw[fill=gray!40] (0-0.5,-0.5-3.1) -- (-0.35-0.5,-1.3-3.1) -- (0.35-0.5,-1.3-3.1) -- cycle;
	\filldraw[fill=gray!40] (0+0.5,-0.5-3.1) -- (-0.35+0.5,-1.3-3.1) -- (0.35+0.5,-1.3-3.1) -- cycle;
	\draw (0-0.5,-1-3.2) node {{\small $A_{S_L}$}};
	\draw (0+0.5,-1-3.2) node {{\small $A_{S_R}$}};
	
	\draw (0,-5.5) node {{\small $\frac{z}{1-2zB} A_{S_L} A_{S_R}$}};
\end{scope}

\begin{scope}[xshift=5.5cm]
	\draw[decorate, decoration=zigzag] (-0.7,-2+0.5) -- (-0.7,-3+0.2);
	\draw (-0.2,-2) node {$=$};
	\draw (1,-1+0.5) -- (1,-4+0.5);
	\fill[black] (1,-1+0.5) circle (1pt);
	\fill[black] (1,-2+0.5) circle (1pt);
	\fill[black] (1,-3+0.5) circle (1pt);
	\fill[black] (1,-4+0.5) circle (1pt);
	\draw (1,-1+0.5) -- (0.5,-1.2+0.5);
	\draw (1,-2+0.5) -- (1.5,-2.2+0.5);
	\draw (1,-3+0.5) -- (1.5,-3.2+0.5);
	\draw (1,-4+0.5) -- (0.5,-4.2+0.5);
	\filldraw[fill=gray!40] (0+0.5,0-0.7) -- (-0.35+0.5,-.75-0.7) -- (0.35+0.5,-0.75-0.7) -- cycle;
	\draw (0+0.5,-0.5-0.75) node {{\small $B$}};	
	\filldraw[fill=gray!40] (0+1.5,0-1.7) -- (-0.35+1.5,-0.75-1.7) -- (0.35+1.5,-0.75-1.7) -- cycle;
	\draw (0+1.5,-0.5-1.75) node {{\small $B$}};
	\filldraw[fill=gray!40] (0+1.5,0-2.7) -- (-0.35+1.5,-0.75-2.7) -- (0.35+1.5,-0.75-2.7) -- cycle;
	\draw (0+1.5,-0.5-2.75) node {{\small $B$}};
	\filldraw[fill=gray!40] (0+0.5,0-3.7) -- (-0.35+0.5,-0.75-3.7) -- (0.35+0.5,-0.75-3.7) -- cycle;
	\draw (0+0.5,-0.5-3.75) node {{\small $B$}};
	
	\draw (0.3,-5.5) node {{\small $\frac{1}{1-2zB}$}};
\end{scope}
\end{tikzpicture}
\end{center}
\caption{\footnotesize 
Sketch of the recursive construction of the generating function $A_{S}(z)$: When $S$ is a Motzkin
tree consisting of more than one vertex (plane binary case), then the first three cases above 
can appear. Here $B$ each time refers to an abstract object representing any tree from
family $\mathcal{B}_n$. If $S$ contains vertices with three or more children, then auxiliary vertices occur. They are embedded according to the fourth picture above and depicted as red squares. 
}
\label{fig:GF_all_plane_recursive}
\end{figure}

Now we show how an embedding of $S$ into $T$ can be constructed in a recursive way - see
Figure~\ref{fig:GF_all_plane_recursive} for a visualization of the used approach.  We start with
the case that the root of $S$ is a unary node. This root has to be embedded at some point in the
tree $T$. The part of $T$ that is above the embedded root of $S$ can be expressed as a path of
left-or-right trees, which contributes a factor $\frac{1}{1-2zB(z)}$. The embedded root of $S$
itself yields a factor $z$, since the generating function of an object of size one is given by
$z$. To the embedded root we have to attach an additional tree $T$ in order to create a binary
structure, yielding a factor $B(z)$, as well as the remaining tree that contains the embedding
of $\tilde{S}$. The factor $2$ that appears in the coefficient in the second case of
\eqref{equ:GF_recusrive_A_S_planebinary} indicates that we work with plane trees - the substructure $\tilde{S}$ can be embedded either in the left or in the right subtree of the unary vertex. 

The third case of \eqref{equ:GF_recusrive_A_S_planebinary}, where $S$ starts with a binary node,
is similar to the previous case. Thus, the factor $\frac{1}{(1-2zB(z))^2}$ corresponds to two
consecutive paths of left-or-right trees, which are separated by the embedded root which itself
gives the additional factor $z$. At some point the lower path splits into two subtrees containing
the embeddings of the subtrees $S_L$ and $S_R$.

By simple iteration one can see that in case of embedding a Motzkin tree $S$, the generating
function $A_S(z)$ reads as 
\begin{align}
\label{equ:GF_embedding_motzkin}
A_S(z) = \left( \frac{z}{(1-2zB(z))^2} \right)^{l-1} \left( \frac{2zB(z)}{1-2zB(z)} \right)^{u} (zB'(z))^l,
\end{align} 
where $l$ denotes the number of leaves and $u$ the number of unary nodes in $S$. The exponent $l-1$ in \eqref{equ:GF_embedding_motzkin} arises from the fact that a Motzkin tree with $l$ leaves has $l-1$ binary nodes, and for each of these nodes we get the respective factor.

Finally, we consider the general case where $S$ is an arbitrary plane tree without any restrictions on the degree distribution sequence. Then we proceed as follows. 
Every $d$-ary node with $d \geq 3$ together with its $d$ children is replaced by a binary tree
having $d$ leaves, which are then replaced by the successors of the original $d$-ary node. 
There are exactly $\boldsymbol{C}_{d-1}$ possible ways to construct such a
binary tree. Unary and binary nodes stay unaltered. Applying this for all nodes results in
constructing a Motzkin tree, called $S'$, and the number of Motzkin trees that can be 
constructed in that way is $\prod_{i = 3}^{m-1} \boldsymbol{C}_{i-1}^{d_i}$. These Motzkin trees
are then embedded with the approach described above. 

\begin{figure}[!h]
\begin{center}
\begin{tikzpicture}
	\draw (-1,-0.4) node {S:};
	\draw (0,0) -- (-0.5,-0.5);
	\draw (0,0) -- (0,-0.5);
	\draw (0,0) -- (0.5,-0.5);
	\draw (-0.5,-0.5) -- (-0.5,-1);
	\fill[red] (0,0) circle (0.1);
	\fill[red] (-0.5,-0.5) circle (0.1);
	\fill[red] (0,-0.5) circle (0.1);
	\fill[red] (0.5,-0.5) circle (0.1);
	\fill[red] (-0.5,-1) circle (0.1);
	
	\node[single arrow, rotate=220, fill=gray!35, minimum height=1.3cm] at (-1.4,-1.4) {};
	\node[single arrow, rotate=320, fill=gray!35, minimum height=1.3cm] at (1.4,-1.4) {};

	\draw (0-2.5,0-2) -- (-0.5-2.5,-0.5-2);
	\draw (0-2.5,0-2) -- (0.5-2.5,-0.5-2);
	\draw (-0.5-2.5,-0.5-2) -- (-0.5-2.5,-1-2);
	\draw (0.5-2.5,-0.5-2) -- (0.2-2.5,-1-2);
	\draw (0.5-2.5,-0.5-2) -- (0.8-2.5,-1-2);	
	\fill[red] (0-2.5,0-2) circle (0.1);
	\fill[red] (-0.5-2.5,-0.5-2) circle (0.1);
	\fill[red] (0.5-2.5-0.1,-0.5-2-0.1) rectangle (0.5-2.5+0.1,-0.5-2+0.1);
	\fill[red] (-0.5-2.5,-1-2) circle (0.1);
	\fill[red] (0.2-2.5,-1-2) circle (0.1);
	\fill[red] (0.8-2.5,-1-2) circle (0.1);
	
	\draw (0+2.5,0-2) -- (-0.5+2.5,-0.5-2);
	\draw (0+2.5,0-2) -- (0.5+2.5,-0.5-2);
	\draw (-0.5+2.5,-0.5-2) -- (-0.8+2.5,-1-2);
	\draw (-0.5+2.5,-0.5-2) -- (-0.2+2.5,-1-2);
	\draw (-0.8+2.5,-1-2) -- (-0.8+2.5,-1.5-2);
	\fill[red] (0+2.5,0-2) circle (0.1);
	\fill[red] (-0.5+2.5-0.1,-0.5-2-0.1) rectangle (-0.5+2.5+0.1,-0.5-2+0.1);
	\fill[red] (0.5+2.5,-0.5-2) circle (0.1);
	\fill[red] (-0.8+2.5,-1-2) circle (0.1);
	\fill[red] (-0.2+2.5,-1-2) circle (0.1);
	\fill[red] (-0.8+2.5,-1.5-2) circle (0.1);

	\node[single arrow, rotate=250, fill=gray!35, minimum height=1.2cm] at (-2.7,-3.8) {};
	\node[single arrow, rotate=290, fill=gray!35, minimum height=1.2cm] at (3,-3.8) {};

\begin{scope}[xshift=-3cm, yshift=-4.6cm]
	\draw[decorate, decoration=zigzag] (0,0) -- (0,-1);
	\draw[decorate, decoration=zigzag] (0,-1) -- (0,-2);
	\fill[red] (0,-1) circle (0.1);
	\fill[black] (0,-2) circle (0.07);
	\draw[decorate, decoration=zigzag] (0,-2) -- (-1,-3);
	\draw[decorate, decoration=zigzag] (0,-2) -- (1,-3);

	\draw (-1,-3) -- (-1-0.5,-0.5-3);
	\draw (-1,-3) -- (-1+0.5,-0.5-3);
	\fill[red] (-1,-3) circle (0.1);
	\filldraw[fill=gray!40] (0-1.5,-0.5-3) -- (-0.35-1.5,-1.3-3) -- (0.35-1.5,-1.3-3) -- cycle;
	\filldraw[fill=gray!40] (0-0.5,-0.5-3) -- (-0.35-0.5,-1.3-3) -- (0.35-0.5,-1.3-3) -- cycle;
	\draw (0-1.5,-1-2.95) node {{\small $B$}};
	\draw (0-0.5,-1-2.95) node {{\small $B$}};
	\fill[red] (-1.4,-4.16) circle (0.1);

	\draw (1,-3) -- (1-0.5,-0.5-3);
	\draw (1,-3) -- (1+0.5,-0.5-3);	
	\fill[red] (1-0.1,-3-0.1) rectangle (1+0.1,-3+0.1);
	
	\filldraw[fill=gray!40] (2-1.5,-0.5-3) -- (2.35-1.5,-1.3-3) -- (2-0.35-1.5,-1.3-3) -- cycle;
	\filldraw[fill=gray!40] (2-0.5,-0.5-3) -- (2.35-0.5,-1.3-3) -- (2-0.35-0.5,-1.3-3) -- cycle;
	\draw (2-1.5,-1-2.95) node {{\small $B$}};
	\draw (2-0.5,-1-2.95) node {{\small $B$}};
	\fill[red] (2-1.4,-4.16) circle (0.1);
	\fill[red] (3-1.4,-4.16) circle (0.1);
\end{scope}

\begin{scope}[xshift=4cm, yshift=-4.6cm]
	\draw[decorate, decoration=zigzag] (0,0) -- (0,-1);
	\draw[decorate, decoration=zigzag] (0,-1) -- (0,-2);
	\fill[red] (0,-1) circle (0.1);
	\fill[black] (0,-2) circle (0.07);

	\draw (0,-2) -- (0.5,-2.5);
	\filldraw[fill=gray!40] (0+0.5,-0.5-2) -- (-0.35+0.5,-1.3-2) -- (0.35+0.5,-1.3-2) -- cycle;
	\draw (0.5,-2.95) node {{\small $B$}};
	\fill[red] (0.6,-3.16) circle (0.1);

	\draw[decorate, decoration=zigzag] (0,-2) -- (-1,-3);
	\draw (-1,-3) -- (-1+0.5,-0.5-3);

	\filldraw[fill=gray!40] (0-0.5,-0.5-3) -- (-0.35-0.5,-1.3-3) -- (0.35-0.5,-1.3-3) -- cycle;
	\draw (0-0.5,-1-2.95) node {{\small $B$}};
	\fill[red] (-0.4,-4.16) circle (0.1);

	\draw[decorate, decoration=zigzag] (-1,-3) -- (-2,-4);
	\fill[red] (-1-0.1,-3-0.1) rectangle (-1+0.1,-3+0.1);	
	
	\draw (-2,-4) -- (-2-0.5,-0.5-4);
	\draw (-2,-4) -- (-2+0.5,-0.5-4);
	\fill[red] (-2,-4) circle (0.1);

	\filldraw[fill=gray!40] (0-2.5,-0.5-4) -- (-0.35-2.5,-1.3-4) -- (0.35-2.5,-1.3-4) -- cycle;
	\filldraw[fill=gray!40] (0-1.5,-0.5-4) -- (-0.35-1.5,-1.3-4) -- (0.35-1.5,-1.3-4) -- cycle;

	\draw (-1-1.5,-1-3.95) node {{\small $B$}};
	\fill[red] (-2.4,-5.16) circle (0.1);

	\draw (0-1.5,-1-3.95) node {{\small $B$}};

\end{scope}
\end{tikzpicture}
\end{center}
\caption{\footnotesize 
Sketch of the principle of embedding an arbitrary plane tree (plane binary case).}
\label{fig:principle_embedding}
\end{figure}
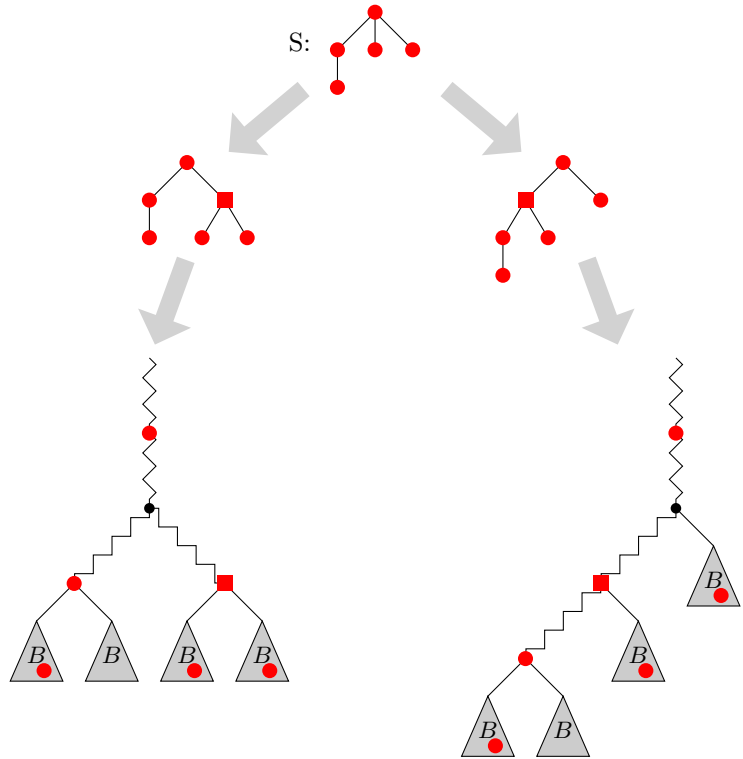
%

But note that replacing a $d$-ary node $v$ with $d \geq 3$ by a binary tree (with more than 1
internal nodes) introduced further vertices into $S$. In particular, $v$ becomes a binary node,
and as there were $d$ successors before, the binary tree arising from $v$ must have $d-1$ internal
nodes, thus giving rise to $d-2$ auxiliary vertices. Since $S$ has $m$ vertices, $S'$ has
therefore $m+\sum_{i=3}^{m-1} (i-2)d_i = 2l+u-1$ vertices.

Furthermore note that the auxiliary vertices have to be embedded into $T$, but they do not belong
to $S$. This causes a special treatment (see Figure~\ref{fig:principle_embedding} for an
illustration). Consider two subtrees $S_1$ and $S_2$ of $S'$ whose last common ancestor (in $S'$)
is an auxiliary vertex, say $w$, and let $v$ be the predecessor of $w$ in $S'$. Of course, any
vertex on the path from $v$ to $w$ may serve as auxiliary vertex instead of $w$. But as $w$ does
not belong to $S$, its actual position is unimportant. Hence, for the sake of not overcounting,
the auxiliary vertices are always placed at the last possible position, see
Figures~\ref{fig:GF_all_plane_recursive} and~\ref{fig:principle_embedding}. This eventually yields
a factor $z/(1-2zB(z))$ for each binary auxiliary node. As there are $2l+u-m-1$ auxiliary nodes
and $m-l-u$ other binary nodes, this gives altogether

\begin{align*}
{A_S}(z) &= \left( \frac{z}{1-2zB(z)} \right)^{2l+u-m-1} \left( \frac{z}{(1-2zB(z))^2}
\right)^{m-l-u} \left( \frac{2zB(z)}{1-2zB(z)} \right)^{u} (zB'(z))^l \prod_{i =
3}^{m-1}(\boldsymbol{C}_{i-1})^{d_i} \\
&= \frac{2^u z^{l+u-1}B(z)^u}{(1-2zB(z))^{m-1}} (zB'(z))^l \prod_{i = 3}^{m-1}
(\boldsymbol{C}_{i-1})^{d_i}  .
\end{align*}
Using the identity $zB'(z) = \frac{B(z)}{1-2zB(z)}$, which holds for plane binary trees, yields the desired result. 
\end{proof}

\begin{ncor}
Let $S$ be a rooted tree. The generating function of the sequence $g_{\mathcal{B}_n}(S)$, which
counts the number of good embeddings of $S$ into all trees of the family $\mathcal{B}_n$ is given by
$$
{G_S}(z) = (1-2zB(z)){A_S}(z).
$$
\end{ncor}

\begin{proof}
The corollary follows immediately, as the only difference in the case of good embeddings is that
the root of $S$ is always embedded in the root of the underlying tree. Thus, we have to omit the path of left-or-right trees in the beginning. This corresponds to a multiplication by the factor $(1-2zB(z))$.
\end{proof}

\section{Asymptotics of the number of embeddings in $\mathcal{B}_n$}  \label{sec_asymptotics}
In this section we investigate the asymptotics of $a_{\mathcal{B}_n}(S)$ and $g_{\mathcal{B}_n}(S)$, as well as monotonicity of their ratio when $S$ is a rooted tree. As a tool we use singularity analysis which provides a relation between the behaviour of a generating function near its dominant singularities (\textit{i.e.} its singularities on the circle of convergence) and the asymptotics of its coefficients. The following lemma will be helpful later on.

\begin{lemma}[Compare Theorems VI.4 and VI.5 in \cite{Flajolet_book}]
\label{l_coeff_asym}
Define 
\begin{align*}
\Delta_0 = \{ z \in \mathbb{C} | |z|<\rho+\epsilon, z\neq \rho, |\rm{arg}(z-\rho)|>\nu \}
\end{align*}
for some $\rho >0, \epsilon>0, 0<\nu<\frac{\pi}{2}$.
Let $r \geq 0$, $\rho_j=\rho e^{i \phi_j}$, for $j = 0,1,\ldots,r$ with $\phi_0=0$ and $\phi_1,\ldots,\phi_r \in (0,2\pi)$. 
Consider $T(z)= \sum_{n \geq 0} T_n z^n$ to be an analytic function in $\Delta := \bigcap_{j=0}^r e^{i \phi_j} \Delta_0$ and satisfying for each $j=0,\ldots,r$
\begin{align*}
T(z) \sim K_j\left( 1- \frac{z}{\rho_j} \right)^{-\alpha_j}, \qquad \text{as} \ z \to \rho_j ~in~ \Delta,
\end{align*}
where $\alpha_j \notin \{ 0,-1,-2,\ldots\}$ and the $K_j$ are constants. Then
\begin{align*}
[z^n]T(z) \sim \sum_{j=0}^{r} K_j \frac{n^{\alpha_j-1}}{\Gamma(\alpha_j)} \rho_j^{-n}, \qquad \text{as} \ n \to \infty.
\end{align*}
\end{lemma}

\bnrem
Note that the assumptions of Lemma~\ref{l_coeff_asym} imply that $\{\rho_0,\rho_1,\dots,\rho_r\}$
is exactly the set of all singularities of the power series $\sum_{n \geq 0} T_n z^n$ on its
circle of convergence. 
\enrem

\begin{thm}
\label{thm_asymptotics}
Consider a rooted tree $S$ with degree distribution sequence $d_S=(l,u,d_2, \dots, d_{m-1})$. Let
$C = \prod_{i=3}^{m-1}(\boldsymbol{C}_{i-1})^{d_i}$. The asymptotics of the number of all
embeddings of $S$ into $\mathcal{B}_n$ is given by
$$
a_{\mathcal{B}_n}(S) \sim \frac{C \cdot 2^{\frac{5-m-3l}{2}}}{\Gamma(\frac{m+l-1}{2})} \cdot 2^n \cdot n^{\frac{m+l-3}{2}}
$$
for $n$ being odd and $a_{\mathcal{B}_n}(S) = 0$ for $n$ being even.
The asymptotics of the number of good embeddings of $S$ into $\mathcal{B}_n$ is given by
$$
g_{\mathcal{B}_n}(S) \sim \left\{  \begin{array}{ll} 
\frac{C \cdot 2^{\frac{6-m-3l}{2}}}{\Gamma(\frac{m+l-2}{2})} \cdot 2^n \cdot n^{\frac{m+l-4}{2}} & \textrm {if} \hspace{10pt} m+l-2 > 0\\
\frac{\sqrt{2} \cdot 2^n}{\sqrt{\pi n^3}} & \textrm{if} \hspace{10pt} m+l-2 = 0
\end{array}\right.
$$
for $n$ being odd and $g_{\mathcal{B}_n}(S) = 0$ for $n$ being even. 
\end{thm}

\begin{proof}
Recall that $a_{\mathcal{B}_n}(S) = [z^n] {A_{S}}(z)$. The function ${A_{S}}(z)$ has two dominant singularities at $\rho_0 = 1/2$ and $\rho_1 = -1/2$.
Expanding ${A_{S}}(z)$ into its Puiseux series at $z \rightarrow \rho_0 = 1/2$ gives
$$
{A_{S}}(z) = C \cdot 2^{\frac{3-m-3l}{2}} \cdot \left(1-\frac{z}{\rho_0}\right)^{-\frac{m+l-1}{2}} \left( 1 + O \left( \left(1-\frac{z}{\rho_0}\right)^{1/2}\right) \right) .
$$
Note that $m+l-1 \geq 1$, since always $l \geq 1$ and $m \geq 1$. Expanding ${A_{S}}(z)$ into a
Puiseux series at $z \rightarrow \rho_1= -1/2$ gives
$$
{A_{S}}(z) = - C \cdot 2^{\frac{3-m-3l}{2}} \cdot \left(1-\frac{z}{\rho_1}\right)^{-\frac{m+l-1}{2}} \left(1 + O \left( \left(1-\frac{z}{\rho_1}\right)^{1/2}\right) \right).
$$
By Lemma \ref{l_coeff_asym} we get
\[
\begin{split}
[z^n] {A_{S}}(z) & \sim \frac{C \cdot 2^{\frac{3-m-3l}{2}}}{\Gamma(\frac{m+l-1}{2})} \cdot (\rho_0)^{-n} \cdot n^{\frac{m+l-3}{2}} - \frac{C \cdot 2^{\frac{3-m-3l}{2}}}{\Gamma(\frac{m+l-1}{2})} \cdot (\rho_1)^{-n} \cdot n^{\frac{m+l-3}{2}} \\
& = \left\{  \begin{array}{ll}
\frac{C \cdot 2^{\frac{5-m-3l}{2}}}{\Gamma(\frac{m+l-1}{2})} \cdot 2^{n} \cdot n^{\frac{m+l-3}{2}} & \textrm {if} \hspace{10pt} $n$ \hspace{5pt} \textrm{is~odd},\\
0 & \textrm{if} \hspace{10pt} $n$ \hspace{5pt} \textrm{is~even}.
\end{array}\right.
\end{split}
\]
The asymptotic analysis for the number of good embeddings is analogous. Again, $g_{\mathcal{B}_n}(S) = [z^n] {G_{S}}(z)$ and ${G_{S}}(z)$ has two dominant singularities at $1/2$ and $-1/2$. For $m+l-2>0$ we obtain
$$
[z^n] {G_{S}}(z) \sim \left\{  \begin{array}{ll}
\frac{C \cdot 2^{\frac{6-m-3l}{2}}}{\Gamma(\frac{m+l-2}{2})} \cdot 2^{n} \cdot n^{\frac{m+l-4}{2}} & \textrm {if} \hspace{10pt} $n$ \hspace{5pt} \textrm{is~odd},\\
0 & \textrm{if} \hspace{10pt} $n$ \hspace{5pt} \textrm{is~even}.
\end{array}\right.
$$
The case $m+l-2=0$ needs to be treated separately. Note that then $m=1$ and $l=1$, thus the structure $S$ that we embed is a single vertex. Therefore the number of good embeddings is just the cardinality of $\mathcal{B}_n$, \textit{i.e.} $g_{\mathcal{B}_n}(S)=\boldsymbol{C}_{\frac{n-1}{2}} \sim \frac{\sqrt{2} \cdot 2^n}{\sqrt{\pi n^3}}$. (Note also that for $S$ being a single vertex $a_{\mathcal{B}_n}(S) = n \boldsymbol{C}_{\frac{n-1}{2}} \sim \frac{\sqrt{2} \cdot 2^n}{\sqrt{\pi n}}$.)
\end{proof}

\begin{ncor}
\label{cor_ratio}
Consider a rooted tree $S$ with degree distribution sequence $d_S=(l,u,d_2, \dots, d_{m-1})$. Let $k = \frac{m+l-2}{2}$ and let $n$ be odd. The asymptotic ratio of the number of good embeddings of $S$ into $\mathcal{B}_n$ to the number of all embeddings into  $\mathcal{B}_n$ is given by
$$
\frac{g_{\mathcal{B}_n}(S)}{a_{\mathcal{B}_n}(S)} \sim \left\{  \begin{array}{ll}
\frac{\Gamma(k+1/2)}{\Gamma(k)} \frac{\sqrt{2}}{\sqrt{n}} & \textrm {if} \hspace{10pt} k > 0,\\
1/n & \textrm {if} \hspace{10pt} k = 0.
\end{array}\right.
$$
\end{ncor}

\begin{proof}
The corollary follows immediately from Theorem \ref{thm_asymptotics}.
\end{proof}

Kubicki~\emph{et al.}~\cite{KLM_ratio_incr} proved that if $T$ is a complete balanced binary tree of arbitrary size and $S_1$, $S_2$ are rooted trees in which each node has at most $2$ descendants (\textit{i.e.}  $S_1$ and $S_2$ are Motzkin trees) and $S_1 \subseteq S_2$, then $\frac{g_{T}(S_1)}{a_{T}(S_1)} \leq \frac{g_{T}(S_2)}{a_{T}(S_2)}$. They also conjectured that the ratio $\frac{g_{T}(S)}{a_{T}(S)}$ is weakly increasing with $S$ for $S$ being any rooted tree. One year later in \cite{KLM_AtoB_is_2tol} they also stated an asymptotic result for the ratio $\frac{g_{T}(S)}{a_{T}(S)}$ when $S$ is an arbitrary rooted tree and $T$ a complete binary tree of size $n$. They showed that $\lim_{n \rightarrow \infty} \frac{g_{T}(S)}{a_{T}(S)} = 2^{l-1}-1$ where $l$ is the number of leaves in $S$. Thereby they proved that for any rooted tree $S$ the asymptotic ratio $\frac{g_{T}(S)}{a_{T}(S)}$ is non-decreasing with $S$ (the function $2^{l-1}-1$ increases with $l$ and if $S_1 \subseteq S_2$ then the number of leaves of $S_2$ equals at least the number of leaves of $S_1$).

The conjecture from \cite{KLM_ratio_incr} was disproved by Georgiou~\cite{Georgiou} who chose specific ternary trees as embedded structures to construct a counterexample. He also generalized the underlying structure to a complete $k$-ary tree and considered strict-order preserving maps instead of embeddings. In this setting he proved that a correlation inequality (corresponding to $\frac{g_{\mathcal{T}_n}(S_1)}{a_{\mathcal{T}_n}(S_1)} \leq \frac{g_{\mathcal{T}_n}(S_2)}{a_{\mathcal{T}_n}(S_2)}$) already holds for $S_1$, $S_2$ being arbitrary rooted trees such that $S_1 \subseteq S_2$.

Referring to the asymptotic result from \cite{KLM_AtoB_is_2tol}, we show below and in the subsequent sections that in our case the asymptotic ratios $\frac{\sqrt{n}~g_{\mathcal{B}_n}(S)}{a_{\mathcal{B}_n}(S)}$, $\frac{\sqrt{n}~g_{\mathcal{T}_n}(S)}{a_{\mathcal{T}_n}(S)}$ and $\frac{\sqrt{n}~g_{\mathcal{V}_n}(S)}{a_{\mathcal{V}_n}(S)}$ are all weakly increasing with $S$ for $S$ being an arbitrary rooted tree. Using this asymptotic result we show later that also the ratios $\frac{g_{\mathcal{B}_n}(S)}{a_{\mathcal{B}_n}(S)}$, $\frac{g_{\mathcal{T}_n}(S)}{a_{\mathcal{T}_n}(S)}$ and $\frac{g_{\mathcal{V}_n}(S)}{a_{\mathcal{V}_n}(S)}$ (unlike in the case from \cite{KLM_ratio_incr}) are eventually weakly increasing with $S$ for sufficiently large $n$. In order to do so, we use Gautschi's inequality given in the following lemma.

\begin{lemma}[Gautschi's inequality, \cite{Gautschi_ineq}]
\label{l_Gautschi}
Let $x$ be a positive real number and let $s \in (0,1)$. Then
$$
x^{1-s} < \frac{\Gamma(x+1)}{\Gamma(x+s)} < (x+1)^{1-s}.
$$
\end{lemma}

\begin{thm}
\label{thm_asm_ratio}
Let $S_1$, $S_2$ be rooted trees such that $S_1 \subseteq S_2$. Then
$$
\lim_{n \rightarrow \infty} \sqrt{n} ~ \frac{g_{\mathcal{B}_n}(S_1)}{a_{\mathcal{B}_n}(S_1)} \leq \lim_{n \rightarrow \infty} \sqrt{n} ~ \frac{g_{\mathcal{B}_n}(S_2)}{a_{\mathcal{B}_n}(S_2)}.
$$
\end{thm}

\begin{proof}
Let $d_{S_1} = (l_1, u_1, \ldots)$, $d_{S_2} = (l_2, u_2, \ldots)$, $k_1 = \frac{m_1+l_1-2}{2}$, $k_2 = \frac{m_2+l_2-2}{2}$ (where $m_i$ denotes the size of $S_i$) and $k_1>0$ (the case when $k_1=0$ is trivial). By Corollary \ref{cor_ratio} we have
$$
\lim_{n \rightarrow \infty} \sqrt{n} ~ \frac{g_{\mathcal{B}_n}(S_1)}{a_{\mathcal{B}_n}(S_1)} = \frac{\sqrt{2} \cdot \Gamma(k_1+1/2)}{\Gamma(k_1)} \hspace{10 pt} \textrm{and} \hspace{10 pt} \lim_{n \rightarrow \infty} \sqrt{n} ~ \frac{g_{\mathcal{B}_n}(S_2)}{a_{\mathcal{B}_n}(S_2)} = \frac{\sqrt{2} \cdot \Gamma(k_2+1/2)}{\Gamma(k_2)}.
$$
Note that the values $k_1$, $k_1+1/2$, $k_2$ and $k_2+1/2$ all belong to the set $\{\frac{1}{2}, 1, \frac{3}{2}, 2, \frac{5}{2}, \ldots\}$. First, we are going to show that the function $f(k) = \frac{\Gamma(k+1/2)}{\Gamma(k)}$ is increasing in $k$ for $k \in \{\frac{1}{2}, 1, \frac{3}{2}, 2, \frac{5}{2}, \ldots\}$. Indeed, applying twice Gautschi's inequality (Lemma \ref{l_Gautschi}) we get for $k>1/2$
$$
\frac{f(k+1/2)}{f(k)} = \frac{\Gamma(k+1)}{\Gamma(k+1/2)} \frac{\Gamma(k)}{\Gamma(k+1/2)} > k^{1/2} (k+1/2)^{1/2}.
$$
Thus, for $k>\frac{\sqrt{17}-1}{4} \approx 0.78$, we obtain $\frac{f(k+1/2)}{f(k)} > 1$. For $k=1/2$ we also have $\frac{f(k+1/2)}{f(k)} = \frac{\pi}{2} > 1$.

Now, it suffices to show that whenever $S_1 \subseteq S_2$, then $k_1 \leq k_2$ (equivalently $m_1+l_1 \leq m_2+l_2$). Of course, $m_1 \leq m_2$. Next, observe that if $S_1 \subseteq S_2$, then also $l_1 \leq l_2$. Indeed, the number of leaves in a tree is the cardinality of its largest antichain. If $S_1$ has $l_1$ leaves and $S_1 \subseteq S_2$, then $S_2$ needs to contain an antichain of cardinality $l_1$ as a subposet, which means that its number of leaves has to satisfy $l_2 \geq l_1$. Together we get $m_1+l_1 \leq m_2+l_2$.
%
\end{proof}

\begin{thm}
\label{thm_ratio}
Let $S_1$, $S_2$ be rooted trees such that $S_1 \subseteq S_2$. Then for sufficiently large $n$
$$
\frac{g_{\mathcal{B}_n}(S_1)}{a_{\mathcal{B}_n}(S_1)} \leq \frac{g_{\mathcal{B}_n}(S_2)}{a_{\mathcal{B}_n}(S_2)}.
$$
\end{thm}

\begin{proof}
Let $d_{S_1} = (l_1, u_1, \ldots)$, $d_{S_2} = (l_2, u_2, \ldots)$, $k_1 = \frac{m_1+l_1-2}{2}$,
$k_2 = \frac{m_2+l_2-2}{2}$. Aiming for a contradiction, assume that $S_1 \subseteq S_2$ and that 
there is an increasing sequence $n_0<n_1<n_2<\dots$ such that
$\frac{g_{\mathcal{B}_n}(S_1)}{a_{\mathcal{B}_n}(S_1)} > \frac{g_{\mathcal{B}_n}(S_2)}
{a_{\mathcal{B}_n}(S_2)}$ for all $n\in\{n_i\,\mid\; i\in\mathbb N\}$. 
Then by Theorem \ref{thm_asm_ratio}
$$
\lim_{n \rightarrow \infty} \sqrt{n} ~ \frac{g_{\mathcal{B}_n}(S_1)}{a_{\mathcal{B}_n}(S_1)} = \lim_{n \rightarrow \infty} \sqrt{n} ~ \frac{g_{\mathcal{B}_n}(S_2)}{a_{\mathcal{B}_n}(S_2)} = \frac{\sqrt{2} \cdot \Gamma(k_1+1/2)}{\Gamma(k_1)} = \frac{\sqrt{2} \cdot \Gamma(k_2+1/2)}{\Gamma(k_2)}.
$$
Recall that the function $f(k) = \frac{\Gamma(k+1/2)}{\Gamma(k)}$ is increasing in $k$ for $k \in \{\frac{1}{2}, 1, \frac{3}{2}, 2, \frac{5}{2}, \ldots\}$ thus the above equality implies $k_1 = k_2$, or equivalently $m_1 + l_1 = m_2 + l_2$. By $S_1 \subseteq S_2$ we have $l_1 \leq l_2$ and $m_1 \leq m_2$ (see the proof of Theorem \ref{thm_asm_ratio}), therefore we get $l_1=l_2$ and $m_1=m_2$. Thus $S_1$ and $S_2$ are isomorphic and $\frac{g_{\mathcal{B}_n}(S_1)}{a_{\mathcal{B}_n}(S_1)} = \frac{g_{\mathcal{B}_n}(S_2)}{a_{\mathcal{B}_n}(S_2)}$ which is a contradiction.
%
%
\end{proof}

\section{Embedding disconnected structures in $\mathcal{B}_n$} \label{sec_disconnected}
In this section we briefly discuss the case of embedding disconnected structures in $\mathcal{B}_n$. Note that in this case all the embeddings must be bad (the underlying structure $T$ has only one maximal element $\mathds{1}_T$; as long as the induced structure is disconnected, we can be sure that it does not contain the root $\mathds{1}_T$).

Assume that $S$ is a forest, \textit{i.e.} a set of rooted trees $S_1, S_2, \ldots, S_r$ ($r \geq
2$) with the degree distribution sequence $d_S = (l,u,d_2,\ldots,d_{m-1})$. The underlying
structure $T$ is connected, thus $S_1, S_2, \ldots, S_r$ always have a common parent in $T$. Let
$\sigma = (\sigma_1, \sigma_2, \ldots, \sigma_r)$ be a permutation of the set $\{1, 2, \ldots,
r\}$. Define $S^{(\sigma)}$ to be a structure constructed as shown in Figure \ref{S_sigma} - we
add an additional vertex $\mathds{1}_{S^{(\sigma)}}$ to $S$, which is a common parent of $S_1,
S_2, \ldots, S_r$ appearing in the order given by $\sigma$. Now, instead of counting the number of
embeddings of $S$ into $T$ we can simply count the numbers of good embeddings of $S^{(\sigma)}$ in $T$ for all permutations $\sigma$ generating non-isomorphic structures $S^{(\sigma)}$ and sum them up. Thus,
$$
a_{\mathcal{B}_n}(S) = \sum_{\sigma \in \Sigma} g_{\mathcal{B}_n}(S^{(\sigma)}),
$$
where $\Sigma$ is a set of permutations of $\{1,2,\ldots,r\}$ such that whenever $\sigma, \tau \in \Sigma$ and $\sigma \neq \tau$ then $S^{(\sigma)}$ and $S^{(\tau)}$ are not isomorphic. Moreover, whenever $\tau$ is a permutation of $\{1,2,\ldots,r\}$ and $\tau \notin \Sigma$ then there exists $\sigma \in \Sigma$ such that $S^{(\sigma)}$ and $S^{(\tau)}$ are isomorphic.

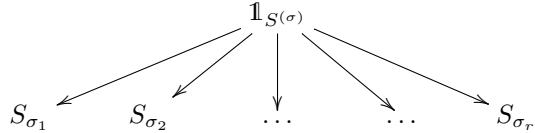
\begin{figure} [!ht]
\center{ 
\begin{displaymath}
    \xymatrix{
			& & \mathds{1}_{S^{(\sigma)}} \ar[dll] \ar[dl] \ar[d] \ar[dr] \ar[drr]  & & \\
			S_{\sigma_1} & S_{\sigma_2} & \ldots & \ldots & S_{\sigma_r}
			}
\end{displaymath}
}
\caption{\footnotesize
The structure of $S^{(\sigma)}$, $\sigma = (\sigma_1, \sigma_2, \ldots, \sigma_r)$.}
\label{S_sigma}
\end{figure}

Note that the asymptotics of $g_{\mathcal{B}_n}(S^{(\sigma)})$ is the same for all $\sigma \in \Sigma$ since the degree distribution sequence of $S^{(\sigma)}$ is the same for all $\sigma \in \Sigma$. It is given by $d_{S^{(\sigma)}} = (\tilde{d}_0, \tilde{d}_1, \ldots, \tilde{d}_{m-1}) = (l,u,\ldots, d_{r-1}, d_r+1, d_{r+1},\ldots,d_{m-1})$. Therefore, by Theorem \ref{thm_asymptotics}
$$
a_{\mathcal{B}_n}(S) \sim \left\{  \begin{array}{ll}
\frac{m!}{k_1! k_2! \ldots k_\ell!} \frac{\tilde{C} \cdot 2^{\frac{6-m-3l}{2}}}{\Gamma(\frac{m+l-2}{2})} \cdot 2^{n} \cdot n^{\frac{m+l-4}{2}} & \textrm {if} \hspace{10pt} $n$ \hspace{5pt} \textrm{is~odd},\\
0 & \textrm{if} \hspace{10pt} $n$ \hspace{5pt} \textrm{is~even}
\end{array}\right.
$$
where $\ell$ is the number of equivalence classes of the set $\{S_1, S_2, \ldots, S_r\}$ with respect to the equivalence relation of being isomorphic and $k_1, k_2, \ldots, k_\ell$ are the cardinalities of those classes. Here $\tilde{C} = \prod_{i=3}^{m-1}(\boldsymbol{C}_{i-1})^{\tilde{d}_i}$. (Note that here we do not consider the case $m+l-2=0$ from Theorem \ref{thm_asymptotics}, because by $r \geq 2$ we always have $m+l-2>0$.)

\section{Non-plane case - embeddings in $\mathcal{V}_n$}
\label{sec_non_plane}

In this section we explain how to take advantage of the results obtained for the plane case in order to infer about the asymptotics of good and all embeddings of a rooted tree $S$ in the family of non-plane binary trees ${\mathcal{V}_n}$.

\begin{thm}
Consider a rooted tree $S$ with degree distribution sequence $d_S=(l,u,d_2, \dots, d_{m-1})$. The
generating function $A_S(z)$ of the sequence $a_{\mathcal{V}_n}(S)$, counting the 
number of all embeddings of $S$ into the family $\mathcal{V}_n$, is given by
\begin{align}
\label{equ:GF_A_S_asymp_thm_nonplane}
A_{S}(z) = \left( \frac{1}{1-zV(z)}\right)^{m+l-1} z^{l+u-1} ~V(z)^{l+u}~ C_{S} ~(1+o(1)) \hspace{1cm} \text{as} \ z \to \pm \rho
\end{align}
where $C_{S}$ is a constant dependent on the structure of $S$ and $V(z)$ is the generating function of the family of non-plane binary trees, satisfying
\begin{align}
\label{equ:GF_non-plane_binary}
V(z) = z + \frac{z}{2}(V(z)^2 + V(z^2)),
\end{align}
which has its dominant singularities at $z = \pm \rho \approx \pm 0.6346$.
\end{thm}

\begin{nrem}
\label{rem:V_and_N}
Note that the value for $\rho$ does not coincide with the one given in \cite[Chapter VII]{Flajolet_book}, since there the size of a tree corresponds to the number of internal nodes, while we count the total number of nodes. This is the reason why in our model the coefficients $V_n$ are zero for even $n$, which yields a periodicity in the generating function that results in the presence of two dominant singularities. 
However, the generating function $N(z)$ of non-plane binary trees where $z$ solely marks the number of internal vertices can easily be connected with our generating function $V(z)$ via $V(z) = z N(z^2)$.
Thus, with the result from \cite{Flajolet_book} that
\begin{align*}
N(z) \sim \frac{1}{\sigma} - a \sqrt{1-\frac{z}{\sigma}}, \hspace{1cm} \text{as} \ z \to \sigma
\end{align*}
with $\sigma \approx 0.4027$ and $a \approx 2.8062$, we immediately know that there are two dominant singularities of $V(z)=zN(z^2)$ at $z = \pm \sqrt{\sigma}$ and we get
\begin{align*}
V(z) = zN(z^2) \sim \pm \sqrt{\sigma} \left( \frac{1}{\sigma} - a \sqrt{2} \sqrt{1 \mp \frac{z}{\sqrt{\sigma}}} \right), \hspace{1cm} \text{as} \ z \to \pm \sqrt{\sigma}.
\end{align*}
Finally, by setting $\rho = \sqrt{\sigma} \approx 0.6346$ and $b = a \sqrt{2 \sigma} \approx 2.5184$ we have
\begin{align*}
V(z) \sim \pm\(\frac{1}{\rho} - b\sqrt{1 \mp \frac{z}{\rho}}\), \hspace{1cm} \text{as} \ z \to \pm \rho.
\end{align*}
\end{nrem}

\begin{proof}
Throughout this proof we write $S_1 \cong S_2$ whenever the structures $S_1$ and $S_2$ are isomorphic. This time we introduce a bivariate  generating function, where $z$ still marks the total number of vertices of a tree, while $u$ is associated with {\it classes of vertices}. Two vertices $v$, $w$ are meant to belong to the same class whenever there exists an isomorphism $f:T \rightarrow T$ such that $f(v) = w$. From \cite{li2016asymptotic} we have 
\begin{align}
\label{equ:V(z,u)}
V(z,u) = zu + \frac{zu}{2} (V(z,u)^2 - V(z^2,u^2) + 2 V(z^2,u)).
\end{align}
By $V_u(z,u)$ we denote the derivative of $V(z,u)$ with respect to $u$, \textit{i.e.} $V_u(z,u) = \frac{\partial V(z,u)}{\partial u}$. We proceed as in the plane case and start with recursively defining the generating function ${A_{S}}(z)$ for the number of embeddings of $S$ into the family $\mathcal{V}_n$, when $S$ is a Motzkin tree:
\begin{align*}
A_{S}(z) = \begin{cases} V_u(z,1) & \quad \text{if} \ S = \bullet \\ 
\frac{ zV(z)}{1-zV(z)} A_{\tilde{S}}(z) & \quad \text{if} \ S = (\bullet, \tilde{S}) \\
\frac{z}{(1-zV(z))^2} {A_{{S_L}}}(z) {A_{{S_R}}}(z) & \quad \text{if} \ S = (\bullet,S_L,S_R) \quad \text{and} \quad S_L \not \cong S_R \\
\frac{z}{(1-zV(z))^2} \frac{1}{2} ({A_{{S_L}}}(z)^2  + {A_{{S_L}}}(z^2)) & \quad \text{if} \ \mathcal{S} = (\bullet,S_L,S_R) \quad \text{and} \quad S_L \cong S_R
\end{cases}.
\end{align*}
The idea of setting up this recursive definition for $A_{S}(z)$ is similar to the plane case with the following differences.
In the first case, corresponding to embedding a single node, we can mark an arbitrary vertex class, instead of an arbitrary vertex, since there might be some non-trivial isomorphisms that would lead to multiple countings of the same embedding.
Furthermore, the paths of left-or-right trees from the previous section, yielding a factor $\frac{1}{1-2zB(z)}$, are now replaced by paths of trees where we do not distinguish between the left-or-right order, since we are in the non-plane setting. Thus, these paths give a factor $\frac{1}{1-zV(z)}$. Finally, in the case when the Motzkin tree starts with a binary root, we have to distinguish between the cases whether the two attached trees are isomorphic or not. The non-isomorphic case works analogously to its plane version, while in the isomorphic case we have to eliminate potential double-countings by using the same idea as for Equation~\eqref{equ:GF_non-plane_binary}.
We do not have to solve the recursion for $A_S(z)$ explicitly, since we are solely interested in the asymptotic behaviour of its coefficients and it is easy to see that asymptotically the contribution of the term $A_{S_L}(z^2)$ is negligible. Since $\rho < 1$ the function $A_S(z^2)$ is analytic at~$z = \rho$. Thus, $[z^n]A_S(z^2) < (\rho+\varepsilon)^{-n}$, which is exponentially smaller than $C \rho^{-n} n^{\beta} = [z^n]A_S(z)$. 

 Thus, by iterating we obtain
\begin{align*}
A_{S}(z) \sim \left( \frac{z}{(1-zV(z))^2} \right)^{l-1} V_u(z,1)^{l} \left( \frac{zV(z)}{1-zV(z)} \right)^{u} \left( \frac{1}{2} \right)^s, \hspace{1cm} \text{as} \ z \to \rho,
\end{align*} 
where $l$ denotes the number of leaves, $u$ the number of unary nodes and $s$ the number of symmetry nodes in $S$ (a symmetry node is a parent of two isomorphic subtrees). An analogous expansion holds for $z\to-\rho$ (only the $1-zV(z)$ in the denominators must be replaced by $1+zV(z)$). 

In the general case where $S$ is an arbitrary non-plane tree, \textit{i.e.} a P\'olya tree, we
proceed as in the previous section and consider the embeddings of all non-plane unary-binary trees
obtained by replacing $d$-ary nodes with $d \geq 3$ together with their children by binary trees
with $d$ leaves. Thus, again taking into account that there are $m-l-u$ binary nodes that were
already there before the replacement (as binary or $d$-ary nodes with $d \geq 3$) and $2l+u-m-1$
auxiliary binary nodes that were introduced by the replacement, we get
\begin{align}
\label{equ:GF_A_S_asymp_nonplane}
A_{S}(z) \sim \left( \frac{z}{1-zV(z)} \right)^{2l+u-m-1}\left( \frac{z}{(1-zV(z))^2}
\right)^{m-l-u} V_u(z,1)^{l} \left( \frac{zV(z)}{1-zV(z)} \right)^{u} C_{S},\ \text{ as} \ z \to \rho,
\end{align} 
and the analogous expansion for $z\to-\rho$. 
The constant $C_S$ arises from the isomorphisms and reads as
\begin{align}
\label{equ:constant_C_s}
C_S = \sum_{\substack{t \in \mathcal{M}_{S} \\ s \ \text{symmetry node of $t$}}} \left( \frac{1}{2} \right)^s,
\end{align}
where $\mathcal{M}_{S}$ denotes the set of all non-plane unary-binary trees obtained from $S$ by
replacing the $d$-ary nodes with non-plane binary trees with $d$ leaves for $d \geq 3$.
Differentiating Equation~\eqref{equ:V(z,u)} with respect to $u$ and plugging $u=1$ yields
\[ V_u(z,1) = \frac{V(z)}{1-zV(z)}. \]
Finally, substituting this expression for $V_u(z,1)$ in Equation~\eqref{equ:GF_A_S_asymp_nonplane} yields the desired result.
Note that the asymptotic equivalence \eqref{equ:GF_A_S_asymp_nonplane}, or
\eqref{equ:GF_A_S_asymp_thm_nonplane} respectively, is also true for the case when $S$ is a single node, \textit{i.e.} $l=1$ and $u=s=0$.
\end{proof}

\begin{thm}
\label{thm_asymptotics_non_plane}
Consider a rooted tree $S$ with degree distribution sequence $d_S=(l,u,d_2, \dots, d_{m-1})$. The
asymptotics of the number of all embeddings of $S$ into $\mathcal{V}_n$ is given by
$$
a_{\mathcal{V}_n}(S) \sim \frac{2C_{S} b^{-m-l+1} \rho^{-m-l}}{ \Gamma(\frac{m+l-1}{2})} \cdot \rho^{-n} \cdot n^{\frac{m+l-3}{2}}
$$
for $n$ being odd and $a_{\mathcal{V}_n}(S) = 0$ for $n$ being even.
The asymptotics of the number of good embeddings of $S$ into $\mathcal{V}_n$ is given by
$$
g_{\mathcal{V}_n}(S) \sim \left\{  \begin{array}{ll} 
\frac{2C_{S} b^{-m-l+2} \rho^{-m-l+1}}{\Gamma(\frac{m+l-2}{2})} \cdot \rho^{-n} \cdot n^{\frac{m+l-4}{2}} & \textrm {if} \hspace{10pt} m+l-2 > 0\\
\frac{b}{\sqrt{\pi}} \cdot \rho^{-n} \cdot n^{-3/2} & \textrm{if} \hspace{10pt} m+l-2 = 0,
\end{array}\right.
$$
for $n$ being odd and $g_{\mathcal{V}_n}(S)=0$ for $n$ being even. Here $b \approx 2.5184$, $\rho \approx 0.6346$ and the constant $C_{S}$, given in \eqref{equ:constant_C_s}, depends on the structure of $S$.
\end{thm}

\begin{proof}
First, note that $V(\rho) \sim \frac{1}{\rho}$, which was already outlined in Remark~\ref{rem:V_and_N}. Therefore, the dominant part of the asymptotics of the coefficients of $A_{S}(z)$ comes from the factors $\frac{1}{1-zV(z)}$, which give
\begin{align*}
\frac{1}{1-zV(z)} \sim \frac{1}{\rho b \sqrt{1-\frac{z}{\rho}}} \hspace{1cm} \text{for} \ z \to \rho.
\end{align*}
The result for $a_{\mathcal{V}_n}(S)$ follows immediately by use of Lemma~\ref{l_coeff_asym}.
As in the plane case, the generating function $G_S(z)$ for the good embeddings just differs from $A_S(z)$ by a factor $(1-zV(z))$ and thus, the asymptotic behaviour of its coefficients can be determined analogously.
Recall that $m+l-2 = 0$ represents the case where $S$ is a single vertex. The number of good embeddings is therefore just the cardinality of $\mathcal{V}_n$ (see Remark~\ref{rem:V_and_N}).
\end{proof}

Now we can formulate a corollary analogous to Corollary \ref{cor_ratio} from the plane case.

\begin{ncor}
\label{cor_ratio_non_plane}
Consider a rooted tree $S$ with degree distribution sequence $d_S=(l,u,d_2, \dots, d_{m-1})$. Let $k = \frac{m+l-2}{2}$ and let $n$ be odd. The asymptotic ratio of the number of good embeddings of $S$ into $\mathcal{V}_n$ to the number of all embeddings into $\mathcal{V}_n$ is given by
$$
\frac{g_{\mathcal{V}_n}(S)}{a_{\mathcal{V}_n}(S)} \sim \left\{  \begin{array}{ll}
\frac{\Gamma(k+1/2)}{\Gamma(k)} \frac{b \rho}{\sqrt{n}} & \textrm {if} \hspace{10pt} k > 0,\\
1/n & \textrm {if} \hspace{10pt} k = 0.
\end{array}\right.
$$
\end{ncor}

\begin{thm}
Let $S_1$, $S_2$ be rooted trees such that $S_1 \subseteq S_2$. Then
$$
\lim_{n \rightarrow \infty} \sqrt{n} ~ \frac{g_{\mathcal{V}_n}(S_1)}{a_{\mathcal{V}_n}(S_1)} \leq \lim_{n \rightarrow \infty} \sqrt{n} ~ \frac{g_{\mathcal{V}_n}(S_2)}{a_{\mathcal{V}_n}(S_2)}.
$$
\end{thm}
\begin{proof}
By Corollary \ref{cor_ratio_non_plane} we get that for any $S$ with $d_S = (l,u,d_2, \ldots, d_{m-1})$ 
$$
\lim_{n \rightarrow \infty} \sqrt{n} ~ \frac{g_{\mathcal{V}_n}(S)}{a_{\mathcal{V}_n}(S)} = \frac{\Gamma(k+1/2)}{\Gamma(k)} b \rho
$$
where $k = \frac{m+l-2}{2} > 0$. The rest of the proof is then analogous to the proof of Theorem \ref{thm_asm_ratio}.
\end{proof}

\begin{ncor}
	Let $S_1$, $S_2$ be rooted trees such that $S_1 \subseteq S_2$. Then for sufficiently large $n$
	$$
	\frac{g_{\mathcal{V}_n}(S_1)}{a_{\mathcal{V}_n}(S_1)} \leq \frac{g_{\mathcal{V}_n}(S_2)}{a_{\mathcal{V}_n}(S_2)}.
	$$
	(Compare Theorem \ref{thm_ratio} and its proof.)
\end{ncor}

Now, let us comment on embedding disconnected structures in a non-plane case. Let $S$ be a forest,
\textit{i.e.} a set of rooted trees $S_1, S_2, \ldots, S_r$, $r \geq 2$. Again, instead of
counting all embeddings of $S$ into $\mathcal{V}_n$, we can count the good embeddings of $\tilde{S}$ in $\mathcal{V}_n$, where $\tilde{S}$ is a forest $S$ with an additional common parent that clips together all $S_i$'s. Note that in the non-plane case the order of $S_i$'s does not matter, thus we simply have
$$
{a}_{\mathcal{V}_n}(S) = {g}_{\mathcal{V}_n}(\tilde{S}).
$$

\section{Planted plane case - embeddings in $\mathcal{T}_n$}
\label{sec_planted_plane}

In this section we extend the results from plane binary trees to planted plane trees, \textit{i.e.} to rooted trees where each internal node can have arbitrarily many child-nodes and the order of the subtrees is important. 
The structures that we embed are as well planted plane trees, and therefore every such a tree $S$ is of the form $S = (\bullet,S_1,\ldots,S_k)$, where the $S_i$'s denote the subtrees that are attached to the root. 
The following lemma contains the construction of the generating function $A_{S}(z)$ of all embeddings of the tree $S$ in the family $\mathcal{T}_n$ of planted plane trees of size $n$.

\begin{lemma}
The generating function $A_{S}(z)$ of all embeddings of $S = (\bullet,S_1,\ldots,S_k)$ into the family $\mathcal{T}_n$ of planted plane trees of size $n$ can be recursively specified as
\begin{align}
\label{equ:GF_plantedplane_As}
A_{S}(z) = \begin{cases} \D zT'(z) = \frac{T(z)(1-T(z))}{1-2T(z)} & \quad \text{if} \ k=0 \\[4mm]
\D\frac{T(z)}{1-2T(z)} {A_{{S_1}}}(z) & \quad \text{if} \ k=1 \\[4mm]
\D\frac{T(z)^2}{(1-2T(z))^2(1-T(z))} {A_{{S_1}}}(z) {A_{{S_2}}}(z) & \quad \text{if} \  k = 2
\\[6mm]
\D\frac{T(z)}{(1-2T(z))^2} \Bigg(  \frac{1-2T(z)}{1-T(z)} A_{S_1}(z) A_{S_{2,k}}(z) \\[5mm]
\D \quad + \frac{T(z)(1-2T(z))}{(1-T(z))^2} A_{S_{1,k-1}}(z) A_{S_k}(z) \\[4mm]
\D \quad + \left( \frac{1-2T(z)}{1-T(z)} \right)^2 \left( A_{S_{1,2}}(z) A_{S_{3,k}}(z) + \ldots
+A_{S_{1,k-2}}(z) A_{S_{k-1,k}}(z) \right) \Bigg) 
& \quad \text{if} \ k > 2
\end{cases}
\end{align}
where $T(z)$ denotes the generating function of the family of planted plane trees, i.e.
\[
T(z) = \frac{1-\sqrt{1-4 z}}{2} = \boldsymbol{C}_0 z + \boldsymbol{C}_1 z^2 + \boldsymbol{C}_2 z^3 + \ldots,
\]
and  $S_{i,j}$ denotes the tree $S_{i,j} = (\bullet,S_i,\ldots,S_j)$ that consists of a root to which the $j-i+1$ subtrees $S_i, \ldots, S_j$ are attached \textnormal{(}in that order\textnormal{)}.
\end{lemma}

\begin{proof}
The case $k=0$ is equivalent to the binary cases, and corresponds to marking an arbitrary node in
the tree $T$. Differentiating both sides of the specification $T(z) = \frac{z}{1-T(z)}$ of planted plane trees with respect to $z$ and solving for $T'(z)$ yields the equality
\[ zT'(z) = \frac{z}{1-2T(z)} = \frac{T(z)(1-T(z))}{1-2T(z)}.\]
Now, let us continue with the proof of the recurrence for the case $k>2$.
In order to do so let us observe Figure \ref{fig:principle_plantedplane} that visualizes how an embedding of a tree $S$ in a tree $T$ can be constructed.
We start with a path of left-or-right plane trees, followed by the embedded root node. Attached to the root node there is another such path, ending with the so-called ``splitting node''. To the left and the right of this second path there can of course be several planted plane trees attached to the embedded root node, which themselves do not contain any embedded vertices.
The two paths that are separated by the embedded root node contribute a factor $\left( \frac{1}{1- \frac{z}{(1-T(z))^2}} \right) ^2$, which can be simplified to $\left( \frac{1-T(z)}{1-2T(z)} \right) ^2$ by means of the functional equation $T(z) = \frac{z}{1-T(z)}$. The root node together with the two sequences of planted plane trees that can be attached to the left or to the right of the path give a factor $\frac{z}{(1-T(z))^2} = \frac{T(z)}{1-T(z)}$.

\begin{figure}
\centering
\includegraphics[scale=0.95]{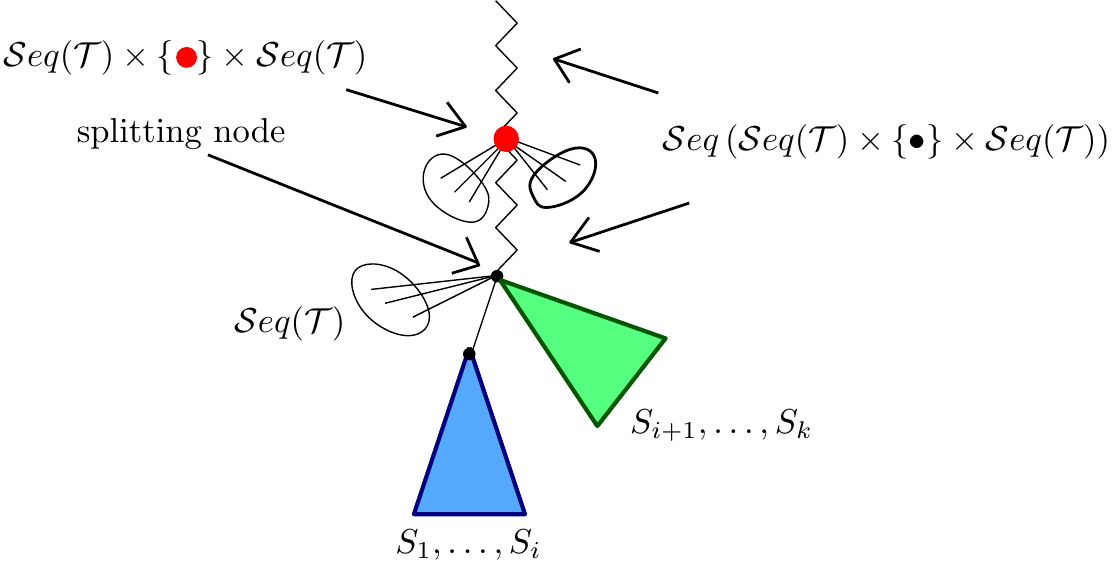}
\caption{\footnotesize
Sketch of the principle of embedding a plane tree $S =(\bullet,S_1,\ldots,S_k)$ into the family of planted plane trees. Here the class $\mathcal{S}eq(\mathcal{T})$ indicates that this part of the structure belongs to this class, i.e., it stands for a sequence of planted plane trees. Likewise,  $\mathcal{S}eq(\mathcal{S}eq(\mathcal{T})\times\{\bullet\}\times\mathcal{S}eq(\mathcal{T}))$ represents a path of left-or-right planted plane trees.}
\label{fig:principle_plantedplane}
\end{figure}

The splitting node can as well have a sequence of plane trees attached, that do not contain any embedded nodes, yielding a factor $\frac{1}{1-T(z)}$, but at some point there has to appear the first plane tree that contains some embedded nodes (pictured in blue in Figure  \ref{fig:principle_plantedplane}). All subtrees attached to the splitting node that are to the right of this blue one are comprised in one plane tree (pictured in green in Figure  \ref{fig:principle_plantedplane}). 
Now we have to distinguish between the cases where a different number of the subtrees $S_1, \ldots, S_{k}$ are embedded in the left (\textit{i.e.} the blue) subtree, while the remaining ones are embedded in the right (\textit{i.e.} the green) tree. These case distinctions give rise to the recursion~\eqref{equ:GF_plantedplane_As} for the generating function. The first two summands of the last case in \eqref{equ:GF_plantedplane_As}, \textit{i.e.} the case $k>2$, represent the cases where one of the $S_i$'s is embedded in a separate subtree:
\begin{itemize}
\item Solely $S_1$ is embedded in the left tree.
In this case we count all embeddings of $S_1$ in the left subtree, giving a factor $A_{S_1}(z)$, while in the right subtree we count exclusively the good embeddings of $S_{2,k} = (\bullet, 
S_2,\ldots, S_k)$, since the splitting node has to be the embedded root of $S_{2,k}$ in order to prevent multiple embeddings of the root.
We already know that the generating function of good embeddings is obtained from the generating function of all embeddings by multiplication with $\frac{1-2T(z)}{1-T(z)}$ (corresponding to 1 divided by the generating function of the starting path) and thus we get the factor $\frac{1-2T(z)}{1-T(z)} A_{S_{2,k}}(z)$.   
\item Solely $S_k$ is embedded in the right tree.
Here we count the good embeddings of $S_{1,k-1}$ in the left tree, as this is general necessary for all cases where we consider more than just one of the $S_i$'s to be embedded in the same subtree. However, in this case we have to count only the bad embeddings of $S_{k}$ in the right tree, since no node of $S$ can be embedded into the splitting node, except the root of $S$, but then the embedding of $S_k$ is still a bad embedding into the green tree. Altogether this yields the factor $\frac{1-2T(z)}{1-T(z)} \frac{T(z)}{1-T(z)} A_{S_{1,k-1}}(z) A_{S_k}(z)$.
\end{itemize}
In all other cases where we embed at least two of the subtrees $S_1, \ldots, S_k$ in both the left
and the right (\textit{i.e.} the blue and the green) subtree, we consider good embeddings for both
subtrees (the blue and the green one, which are then becoming two trees where we embed $S_{1,i}$
and $S_{i+1,k}$, respectively) yielding a factor $\left( \frac{1-2T(z)}{1-T(z)} \right)^2$.
Together with the factors from the two paths, the embedded root and the sequence of plane trees we get the desired coefficients.

The cases $k=1$ and $k=2$ can be treated in the exact same way as we just did for $k>2$. However, note that in the case $k=1$ the green (\textit{i.e.} the right) tree and two of the sequences of planted plane trees are merged such that we end up with one path, the embedded root of $S$ together with its two sequences of planted plane trees and finally the attached blue tree that contains the embedding of the only subtree $S_1$. This yields the factor
\[ \frac{1-T(z)}{1-2T(z)} \frac{T(z)}{1-T(z)} A_{S_1}(z).\]
In the case $k=2$ we have the pre-factor $\frac{T(z)}{(1-2T(z))^2}$ that covers the two paths, the embedded root node with its attached sequences of plane trees and the sequence of plane trees that is attached to the splitting node. Now there is just one splitting option: $S_1$ has to embedded in the left tree, where we consider all embeddings, and $S_2$ has to be embedded in the right tree, where we solely count the bad embeddings of $S_2$, since the splitting node must not be an embedded node. It is easy to verify that this case gives the factor
\[ \frac{T(z)^2}{(1-2T(z))^2(1-T(z))} {A_{{S_1}}}(z) {A_{{S_2}}}(z).\qedhere \]
\end{proof}

\begin{nrem}\label{rem_6}
Note that for the cases $k=0,1,2$ the generating function $A_S(z)$ of all embeddings of $S = 
(\bullet,S_1,\ldots,S_k)$ into the family $\mathcal{T}_n$ of planted planes trees of size $n$ given in \eqref{equ:GF_plantedplane_As} is of the form $f_k(T) \cdot A_{S_1}(z) \ldots A_{S_k}(z)$, where $f_k(T)$ is a function that depends only on $T(z)$ and on the size $k$ of $S$, but not on the specific shape of $S$. 
We want to emphasize that, by digging into the structure of $S$ and by recursive application of the formulas given in \eqref{equ:GF_plantedplane_As}, it follows that $A_S(z)$ is in fact of the form
\[ A_S(z) = f(T) \cdot A_{S_1}(z) \cdots A_{S_k}(z), \]
for arbitrary $S = (\bullet,S_1,\ldots,S_k)$.
 \end{nrem}

Now, we are in the position to obtain the asymptotic number of all and good embeddings of a given plane tree $S$ in the family of planted plane trees.

\begin{thm}
Consider a rooted tree $S$ of size $m$ with degree distribution sequence $d_S=(l,d_1,d_2, \dots,
d_{m-1})$. Let $C = \prod_{i=1}^{m-1} (\boldsymbol{C}_{i-1})^{d_i}$. The asymptotics of the number
of all embeddings of $S$ into $\mathcal{T}_n$ is given by
\begin{align*}
a_{\mathcal{T}_n}(S) \sim \frac{C \cdot (\frac{1}{2})^{m+l}}{\Gamma(\frac{m+l-1}{2})} \cdot 4^n \cdot n^{\frac{m+l-3}{2}}.
\end{align*}
The asymptotics of the number of good embeddings of $S$ into $\mathcal{T}_n$ is given by
\begin{align*}
g_{\mathcal{T}_n}(S) \sim \frac{2 C \cdot (\frac{1}{2})^{m+l}}{\Gamma(\frac{m+l-2}{2})} \cdot 4^n \cdot n^{\frac{m+l-4}{2}}.
\end{align*}
\end{thm}

\begin{proof}
Triggered by the observation in Remark~\ref{rem_6}, let us set 
\begin{equation}\label{f_rec}
f_1(z) = \frac{1}{2(1-2T(z))},\ \text{ and } \ f_k(z) = \frac{A_{S}(z)}{\prod_{i=1}^k
A_{S_i}(z)} \ \text{ for } k >1.
\end{equation}
Then \eqref{equ:GF_plantedplane_As} immediately gives $f_2(z) = T(z)^2/((1-2T(z))^2(1-T(z)))$.

Next, consider the last equation of \eqref{equ:GF_plantedplane_As} (the case $k\ge 3$) and observe
that all generating functions on the right-hand side which are associated with a composite
structure are of the form $A_{S_{i,j}}(z)$, where the root of $S_{i,j}$ has degree at least two. Thus,
dividing the equation by $\prod_{i=1}^k A_{S_i}(z)$ (and cancelling out all single $A_{S_j}(z)$) 
yields only quotients which can be readily turned into $f_\ell(z)$ with suitable choices of
$\ell$, because the case $\ell=1$ does not appear here. A straight-forward simplification then
gives 
\begin{align}
\label{equ:f_k_rec}
f_k(z) = 
\frac{T(z)}{(1-T(z))^2} \sum_{j=1}^{k-1} f_j(z) f_{k-j}(z) \ \text{ for } k \geq 3.
\end{align}
Both sides of this equation tend to infinity, as $z\to 1/4$, and we need their singular behaviour
for our analysis of $A_S(z)$. Hence, we set $g_k(z)=(1-2T(z))^k f_k(z)$ for $k\ge 1$. Plugging
this into \eqref{equ:f_k_rec} we observe that $g_k(z)$ satisfies the same recurrence as $f_k(z)$,
but with the initial values $g_1(z)=1/2$ and $g_2(z)=T(z)^2/(1-T(z))$. As $T(1/4)=1/2$, the
functions $g_k(z)$ are regular at $z=1/4$. 
By evaluating the recurrence at $z=1/4$ and setting $h_k := 2g_k(1/4)$, we get 
a recurrence for $h_k$, which is in fact already valid for $k\ge2$: 
\begin{align*}
h_1=1 \quad \text{ and } \quad h_k = \sum_{j=1}^{k-1} h_j h_{k-j} \ \text{ for } k\ge 2. 
\end{align*}
This is exactly the recurrence for the Catalan numbers, and thus, $h_k = \boldsymbol{C}_{k-1}$.

Hence, for $z \to 1/4$ and $k \geq 2$ we have
\begin{align*}
f_k(z) \sim \frac{1}{2} \boldsymbol{C}_{k-1} (1-4z)^{-k/2},
\end{align*}
which implies that as $z\to 1/4$ we have
\begin{align*}
A_{S}(z) \sim \frac{\boldsymbol{C}_{k-1}}{2} (1-4z)^{-k/2} A_{S_1}(z) \ldots A_{S_k}(z) = \left( \prod_{i=1}^{m-1} \left( \frac{\boldsymbol{C}_{i-1}}{2} (1-4z)^{-i/2} \right)^{d_i} \right) \left( A_{\bullet} (z)\right)^l,
\end{align*}
where $S = (\bullet,S_1,\ldots,S_k)$, $d_i$ denotes the number of nodes with out-degree $i$, $l$
denotes the number of leaves, \textit{i.e.} $l=d_0$, and the equation follows from recursively
going into the subtrees $S_1,\ldots,S_k$ and using \eqref{f_rec} until one encounters a leaf of
$S$. Then each leaf yields a factor $A_{\bullet} (z)$.  
Using the equality $A_{\bullet}(z) = zT'(z)\sim \frac12(1-4z)^{-1/2}$, which follows from
\eqref{equ:GF_plantedplane_As} and the singular expansion of $T(z)$, we get for $ z \to \frac{1}{4}$
\begin{align}
\label{equ:asymp_expansion_As}
A_{S}(z) \sim \left( \prod_{i=1}^{m-1} \left( \frac{\boldsymbol{C}_{i-1}}{2} \right)^{d_i} \right) (1-4z)^{-\left(l + \sum_{i=1}^{m-1} id_i \right)/2} \left( \frac{1}{4} \right)^l.  
\end{align}
Note that $\sum_{i=1}^{m-1} id_i = m-1$, since every vertex with out-degree $i$ is counted exactly $i$ times and thus, we simply obtain the total number of nodes with in-degree greater than zero (\textit{i.e.} all nodes except for the root). We also have $\sum_{i=1}^{m-1} d_i = m-l$ thus (recall that $C = \prod_{i=1}^{m-1} (\boldsymbol{C}_{i-1})^{d_i}$)
\[
\left( \frac{1}{4} \right)^l \cdot \prod_{i=1}^{m-1} \left( \frac{\boldsymbol{C}_{i-1}}{2} \right)^{d_i} = C \left( \frac{1}{2} \right)^{m+l}.
\]
Finally, Lemma~\ref{l_coeff_asym} gives
\begin{align*}
a_{\mathcal{T}_n}(S) \sim \frac{C \cdot (\frac{1}{2})^{m+l}}{\Gamma(\frac{m+l-1}{2})} \cdot 4^n \cdot n^{\frac{m+l-3}{2}}.
\end{align*}

The generating function of the number of good embeddings can be derived from the generating function $A_{S}(z)$ by multiplication by the factor $\frac{1-2T(z)}{1-T(z)}$. This factor is responsible for getting rid of the path of trees which could appear above embedded root of $S$ when we were considering all embeddings. Thus we have $G_{S}(z) = \frac{1-2T(z)}{1-T(z)} A_{S}(z)$. Noticing that 
\begin{align*}
\frac{1-2T(z)}{1-T(z)} = 2 \frac{\sqrt{1-4z}}{1+\sqrt{1-4z}},
\end{align*}
using \eqref{equ:asymp_expansion_As} and applying Lemma~\ref{l_coeff_asym} yields the desired result.
\end{proof}


\begin{ncor} \label{cor_ratio_plane_planted}
Consider a rooted tree $S$ of size $m$ with $l$ leaves. The asymptotic ratio of the number of good
embeddings of $S$ into $\mathcal{T}_n$ to the number of all embeddings in  $\mathcal{T}_n$ is given by
$$
\frac{g_{\mathcal{T}_n}(S)}{a_{\mathcal{T}_n}(S)} \sim \left\{  \begin{array}{ll}
\frac{2 \Gamma(\frac{m+l-1}{2})}{\Gamma(\frac{m+l-2}{2}) \sqrt{n}} & \textrm {if} \hspace{10pt} m > 1,\\
1/n & \textrm {if} \hspace{10pt} m = 1.
\end{array}\right.
$$
\end{ncor}
\begin{thm}
Let $S_1$, $S_2$ be rooted trees such that $S_1 \subseteq S_2$. Then
$$
\lim_{n \rightarrow \infty} \sqrt{n} ~ \frac{g_{\mathcal{T}_n}(S_1)}{a_{\mathcal{T}_n}(S_1)} \leq \lim_{n \rightarrow \infty} \sqrt{n} ~ \frac{g_{\mathcal{T}_n}(S_2)}{a_{\mathcal{T}_n}(S_2)}.
$$
\end{thm}
\begin{proof}
By Corollary \ref{cor_ratio_plane_planted} we get that for any $S$ with $d_S = (l,u,d_2, \ldots, d_{m-1})$ 
$$
\lim_{n \rightarrow \infty} \sqrt{n} ~ \frac{g_{\mathcal{T}_n}(S)}{a_{\mathcal{T}_n}(S)} = \frac{2 \Gamma(k+1/2)}{\Gamma(k)}.
$$
where $k = \frac{m+l-2}{2} > 0$. The rest of the proof is then analogous to the proof of Theorem \ref{thm_asm_ratio}.
\end{proof}
\begin{ncor}
	Let $S_1$, $S_2$ be rooted trees such that $S_1 \subseteq S_2$. Then for sufficiently large $n$
	$$
	\frac{g_{\mathcal{T}_n}(S_1)}{a_{\mathcal{T}_n}(S_1)} \leq \frac{g_{\mathcal{T}_n}(S_2)}{a_{\mathcal{T}_n}(S_2)}.
	$$
	(Compare Theorem \ref{thm_ratio} and its proof.)
\end{ncor}

\section{Discussion}  \label{sec_discussion}
We proved that the ratio of the number of good embeddings to the number of all embeddings of a
given tree $S = (\bullet,S_1,\ldots,S_k)$ into the families of trees
$\mathcal{B}_n, \mathcal{V}_n, \mathcal{T}_n$ is asymptotically of the same order for all the
three considered families of trees, namely plane binary trees, non-plane binary trees and planted
plane trees. Thereby we extended the results of Kubicki \emph{et
al.}~\cite{KLM_ratio_incr,KLM_AtoB_is_2tol} and Georgiou~\cite{Georgiou}. We expect that this
result will also hold for the family of P\'olya trees, which are the closest counterpart to posets
that admit a (rooted) treelike shape, \textit{i.e.} they have a single maximal element. In
principle, the approach that we used within this paper works for embeddings into the family of
P\'olya trees as well. However, one would have to consider all possible partitions of $S_1,
\ldots, S_k$, as any collection of isomorphic subtrees within $S_1,\ldots, S_k$ admits non-trivial
isomorphisms between the $S_i$'s, which can get rather involved and is therefore omitted in this
work. 


\bibliographystyle{plain}
\bibliography{tree_embeddings}

\end{document}